\documentclass[13pt]{amsart}
\usepackage[left=1.5in, right=1.5in]{geometry}

\usepackage[colorlinks=true]{hyperref}

\usepackage[latin1]{inputenc}
\usepackage{amsfonts}
\usepackage[toc,page,title,titletoc,header]{appendix}
\usepackage{graphicx,psfrag,epsfig,multirow,hyperref,mathrsfs}
\usepackage{amssymb,amsmath,amscd,amsthm,verbatim}
\newtheorem{Thm}{Theorem}[section]
\newtheorem{Lm}[Thm]{Lemma}
\newtheorem{Prop}[Thm]{Proposition}

\newtheorem{Def}[Thm]{Definition}
\newtheorem{Rk}{Remark}[section]

\newcommand{\al}{\alpha}

\newcommand{\cR}{\mathcal{R}}

\newcommand{\cG}{\mathcal{G}}
\newcommand{\cF}{\mathcal{F}}

\newcommand{\cK}{\mathcal{K}}

\newcommand{\cN}{\mathcal{N}}
\newcommand{\cM}{\mathcal{M}}
\newcommand{\cS}{\mathcal{S}}

\newcommand{\R}{\mathbb{R}}
\newcommand{\Z}{\mathbb{Z}}

\newcommand{\N}{\mathbb{N}}

\newcommand{\dt}{\delta}
\newcommand{\lb}{\lambda}

\newcommand{\eps}{\varepsilon}
\newcommand{\T}{\mathbb{T}}

\def\bk{\mathbf{k}}
\def\bp{\mathbf{p}}

\title{Closing lemma and KAM normal form}
\author{Jinxin Xue}
\email{jxue@tsinghua.edu.cn}
\address{Department of Mathematics, Tsinghua University, Beijing, China, 100084. }
\begin{document}
\maketitle
\begin{abstract}
	In this paper, we develop an approach to the problem of closing lemma based on KAM normal form. The new approach differs from existing $C^1$ perturbation approach and spectral approach, and can handle the high regularity, high dimensional cases and even Riemannian metric perturbations. Moreover, the proof is constructive and effective. We apply the method to the original nearly integrable setting of Poincar\'e and confirm several old and new conjectures with weak formulations. First, for Poincar\'e's original setting of nearly integrable systems, we prove that for typical perturbations, periodic orbits are asymptotically dense as the size of perturbation tends to zero. Second, we prove that typical smooth perturbation of the geodesic flow on the flat torus has asymptotically dense periodic orbits, which partially solves an open problem since Pugh-Robinson's $C^1$-closing lemma. Third, we prove that for typical Hamiltonian or contact perturbation of the geodesic flows of the ellipsoid has asymptotically dense orbit on the energy level, which  enhances the recent researches on strong closing lemma, and also confirms partially a conjecture of Fish-Hofer in this setting. We also discuss the relation of our models to the recent researches on many-body localization in physics. 
\end{abstract}
\tableofcontents

\addtocontents{toc}{\setcounter{tocdepth}{1}} 
\section{Introduction}
In this paper, we study the problem of closing lemma in the original setting of Poincar\'e. In his work on the $N$-body problem, Poincar\'e intensively searched for periodic orbits. He explained his reasons as follows (c.f. Section 36 of Chapter III of \cite{P}):

{\it Here is a fact that I could not prove rigorously, but which nevertheless seems very likely to me. Given equations of the form defined in Section 13 and an arbitrary solution of these equations, one can always find a periodic solution $($with a period which, admitedly, may be very long$)$, such that the difference between the two solutions is arbitrarily small. In fact, what makes these solutions so precious to us, is that they are, so to say, the only opening through which we can try to enter a place which, hitherto, was deemed inaccessible.}

Here the equations under consideration are the canonical equations of nearly integrable Hamiltonian of the form
\begin{equation}\label{EqHam}
H_\eps(I,\theta)=h(I)+\eps f(I,\theta), \quad (I,\theta)\in T^*\T^N	
\end{equation}
used to model the $N$-body problem. The last paragraph is the origin of the problem of closing lemma. 

\subsection{The torus case and KAM theory}
The modern formulation of the conjecture of closing lemma states that \emph{a nonwandering point can be closed into a periodic point by an arbitrarily small perturbation of the system}, where a nonwandering point is a point which does not have a neighborhood disjoint from its images under iterates. This, combined with Poincar\'e recurrence theorem, would imply the denseness of periodic orbits. With this formulation, it was proved by Pugh \cite{P1} that closing lemma holds for $C^1$ generic diffeomorphisms, and later with Robinson \cite{PR} for $C^2$ generic Hamiltonian systems. We refer readers to an elegant review \cite{P2} by Pugh for some historical remarks.  However,  this approach is only local, and it is known that it fails for function spaces with higher regularity. 

The work of Pugh and Robinson left open the following problem (c.f. Section 10 of \cite{PR})

{\bf Problem:} {\it Prove the closing lemma for Riemannian metric perturbation of geodesic flows.}

The difficulty of the problem is that Riemannian metric perturbations cannot be local, since a perturbation on the base manifold will simultaneously perturb the tangent bundle. 
In \cite{Ri}, Rifford proved closing lemma of geodesic flows for $C^1$ generic Riemannian metrics. Note that the regularity is lower than the $C^2$ Hamiltonian perturbations of \cite{PR}. 

Our first result proves a  version of $C^r$ closing lemma for typical Riemannian metric perturbation of the flat torus $\T^N$ for any $r$ large.  We start with a formulation of the typicality. 

We consider Riemannian metric $ds^2=\sum_{i=1}^N d\theta_i^2$ on the torus $\T^N$ and its perturbation $ds^2_{\eps A}=\sum_{i,j=1}^N (\dt_{ij}+\eps a_{ij}(\theta))d\theta_id\theta_j. $ Let us denote by $\mathrm{Sym}^r(\T^N)$ the space of $N\times N$ symmetric matrices each of whose entries is in $C^r(\T^N)$. For $A:=(a_{ij})\in \mathrm{Sym}^r(\T^N)$, we introduce its norm as $\|A\|:=\sum_{i\leq j} \|a_{ij}\|_{C^r}$. We also denote by $\cS^r$ the unit sphere in $\mathrm{Sym}^r(\T^N)$.

\begin{Def}
	We say that a sequence of sets $\cR_n\subset \mathcal S^r$ is \emph{asymptotically residual}, if there is a sequence $c_n\to0+$ and a sequence of first category set $\cN_n\subset\mathcal S^r$ such that we have $\cR_n\cup B^2_{c_n}(\cN_n)=\mathcal S^r$ for all $n\in \N$  where $B^2_{c_n}(\cN_n)$ means the union of balls of radius $c_n$ in the $C^2$-norm centered on $\cN_n$. 
\end{Def}

\begin{Thm}\label{ThmRm}
	There exists $r_0>0$ such that for all $r>r_0,$ the following holds. 
	There is a sequence of asymptotically residual set $\cR_n\subset \cS^r$ and a sequence of $\eps_n\to 0$, such that for  any $\dt>0$, there exists $n_\dt$ such that for any $\eps<\eps_{n_\dt}$, any $(a_{ij})\in \cR_{n_\dt}$ and any ball $B_\dt$ of radius $\dt$ in the unit tangent bundle, the  Riemannian metric $ds^2_{\eps A}$ has a closed geodesic intersecting  $B_\dt. $ 
\end{Thm}
As we will see in the proof, we can make the dependence of $\eps_{n_\dt}$ on $\dt$ explicit, for instance $\eps_{n_\dt}=\dt^{10}$, thus the proof is effective. 

We illustrate the idea of the asymptotically residual set using an example. It is known that Morse functions (functions all whose critical points are nondegenerate) are generic in   $C^r(M)$, $r\geq 2$, where $M$ is a compact manifold. For a sequence $\mathbf c=\{c_n\}$ with $c_n\to 0+$, we introduce $\cN_n$ to be the set of functions that are not Morse, i.e. some critical points are degenerate and $\mathcal R_{n}$  to be the set of $C^r$ functions such that for all $f\in \cR_n$, we have the absolute values of the eigenvalues of $\mathrm{Hess}f$ at each critical point are bounded away from $0$ by $c_n^2$. A function on the boundary of $\cR_n$ can be perturbed into one with degenerate critical point, by a perturbation of $C^2$-norm at most $c_n$ supported in an $O(1)$-neighborhood of a critical point.

The method of the proof is the normal form theory in Hamiltonian dynamics. Though our method is perturbative, it differs drastically from the approach of closing a nonwandering point. Indeed, our approach is to show that certain periodic orbits in the unperturbed system persist in the perturbed one.  Generalizing the proof of the last theorem, we get the following version of closing lemma in Poincar\'e's original setting.

Let $U\subset \R^N$ be a bounded and open subset and $\bar U$ its closure. We consider $C^r(\bar U\times \T^N)$ the space  $C^r$ functions on $\bar U\times \T^N$. Denote by $\mathcal S^r(\bar U\times \T^N)\subset C^r(\bar U\times \T^N)$ its unit sphere consisting of functions with $\|f\|_{C^r}=1$. Similarly, we define $C^r(\bar U)$ and $\cS^r(\T^N)$. 

\begin{Thm}\label{ThmTorus}	
		There exists $r_0>0$ such that for all $r>r_0,$ the following holds. There is a residual set $\cR\subset C^r(\bar U)$, such that for each $h\in \cR$, there exists a sequence of asymptotically residual set $\cR_n\subset \cS^r(\bar U\times \T^N)$ $($or $\cS^r( \T^N))$ and a sequence of $\eps_n\to 0$, such that for  any $\dt>0$, there exists $n_\dt$ such that for any $\eps<\eps_{n_\dt}$, any $f\in \cR_{n_\dt}$ and any ball $B_\dt$ of radius $\dt$ in $T^*\T^N$ centered on $U\cap h^{-1}(1)$, the Hamiltonian system \eqref{EqHam}  has a periodic orbit intersecting  $B_\dt. $ 
	
\end{Thm}

Here the first residual set $\cR$ is chosen to guarantee that $h\in \cR$ has nondegenerate Hessian for almost every point in $h^{-1}(1)$, i.e.  the classical Kolmorogov's nondegeneracy is satisfied almost everywhere. It is standard to assume certain nondegeneracy conditions in both KAM theory (\cite{K,A,M}) and Nekhoroshev Theorem (\cite{N,L}). The degenerate cases are also studied in literature with in particular the prototypical unperturbed Hamiltonian $h(I)=\omega\cdot I$ where $\omega\in \R^N$, mostly with Diophantine condition on $\omega$. If $I$ is small, the relevant problem is to study the stability of elliptic fixed point.  It is known that there is a positive measure set of Lagrangian invariant tori accumulating to the elliptic fixed point in the analytic category \cite{EFK2} and that the elliptic fixed point may be isolated in the smooth category \cite{FS}. If $I$ is not small, then the relevant problem is the stability of KAM tori. Similar results were obtained in \cite{EFK1}.  There is also work on effective long time stability of the elliptic fixed point \cite{Ni}.

\subsection{The ellipsoid case and strong closing lemma}As Pugh's $C^1$ closing lemma has been well-known for a long time,  it came as a big surprise when Irie proved the denseness of periodic orbits for $C^\infty$-generic Reeb flows (c.f. \cite{I1}) on a 3-dimensional contact manifold and later for $C^\infty$-generic surface Hamiltomorphisms (c.f. \cite{AI}). The approach of Irie is global. He uses the spectral information given by ECH. The idea turns out to be very fruitful. For instance, similar idea was used in the min-max theory of minimal surfaces solving a conjecture of Yau (c.f. \cite{IMN}) and proving an equidistribution result (c.f. \cite{MNS}). See also \cite{CPZ,EH} etc. 

Then it is natural to pursue this stream of ideas in the higher dimensional case. For instance, it is conjectured in Fish and Hofer that periodic orbits are dense, for instance, for the Reeb flow of $C^\infty$-generic contact hypersurfaces in $\R^{2N}$.  We refer readers to Conjecture 1 of \cite{FH} for precise statement, speculations and more references. 


In a recent paper \cite{I2}, Irie introduced a notion of \emph{strong closing property} formlated as follows. Let $M$ be a manifold of dimension $2N-1$ and $\lambda$ a contact form that is a one-form satisfying $\lambda \wedge (d\lambda)^{N-1}\neq 0$ at each point of $M$. This defines a Reeb vector field $R_\lambda$ via $\iota_{R_\lambda}d\lambda\equiv0$ and $\lambda(R_\lambda)\equiv1$.  The \emph{positive strong closing property} is defined as: \emph{for all $\psi\in C^\infty(M,\R_{\geq 0})$ nonconstant, there exists $t\in [0,1]$ such that there is a closed Reeb orbit for the contact manifold $(M, (1+t \psi)\lambda)$ that intersects the support of $\psi$. }

Though the strong closing property is weaker than existing dense periodic orbits, to our understanding, it is strong in the sense that a periodic orbit is created by a perturbation $\psi$, no matter how small the support of $\psi$ is. 

Irie also conjectured that the strong closing property holds for $(\partial E_a, \lambda)$ where  $\lambda=\frac{1}{2}\sum_{i=1}^N y_idx_i-x_idy_i$ is the standard contact form and $E_a$ is the ellipoid
$$E_a=\left\{(x_1,\ldots, x_n, y_1,\ldots, y_N)\in \R^{2N}\ |\ \sum_{i=1}^N 
\frac{\pi (x_i^2+y_i^2)}{a_i}\leq 1\right\}$$
where $a=(a_1,\ldots, a_N)$ satisfies $0<a_1\leq a_2\leq \ldots\leq a_N$. 
The conjecture was proved by Chaidez-Datta-Prasad-Tanny\cite{CDPT} using tools from contact homology and by Cineli and Seyfaddini \cite{CS} using spectral invariants. The latter also considers strong closing property for some very chaotic models such as Anosov-Katok construction for pseudo rotations. Note that the Reeb flow of $(\partial E_a, \lambda)$ is the same as a Hamiltonian flow of the Hamiltonian 
\begin{equation}\label{EqHam0}
	H(x,y)=\sum_{i=1}^N 
	\frac{\pi (x_i^2+y_i^2)}{a_i}	
\end{equation}
on the energy level 1. A small perturbation of the contact form induces a perturbed Reeb flow that can be viewed as a Hamiltonian system slightly perturbing the Hamiltonian \eqref{EqHam0}. It is standard to introduce a symplectic transformation $(x_i,y_i)=(\sqrt I_i \cos\theta_i, \sqrt I_i\sin \theta_i)$. Thus we model the problem as a nearly integrable system 
\begin{equation}\label{EqHamMain}
	H_\eps (I,\theta)=\omega\cdot I+\eps f(I,\theta), \quad (I,\theta)\in T^*\T^N,	
\end{equation}
where $\omega=(\omega_1,\ldots, \omega_N)$ and $\omega_i=\frac{\pi}{a_i},\ i=1,\ldots, N. $ 

Note that this is exactly the prototypical degenerate case of the KAM theory and Nekhoroshev theorem mentioned above. 
We have the following results concerning Irie's conjecture. We state our results in the language of nearly integrable Hamiltonian dynamics, which have counterparts for Reeb flows by the correspondence in Appendix \ref{SRelation}. 

Let $C_{c}^r(T^*\T^N)$ be the space of compactly supported $C^r$ functions on $T^*\T^N$. We first perform a splitting $C_{c}^r(T^*\T^N)=C_{c}^r(\R^N)\oplus (C_{c}^r(T^*\T^N)/C_{c}^r(\R^N))$ such that each $f\in C_{c}^r(T^*\T^N)$ is split into $f=\langle f\rangle+[f]$ where $\langle f\rangle(I)=\int_{\T^N}f(I,\theta)d\theta\in C_{c}^r(\R^N)$ and $[f]=f-\langle f\rangle\in C_{c}^r(T^*\T^N)/C_{c}^r(\R^N). $

Let $D\subset \Delta^{N-1}:=\{\omega\cdot I=1,\ I_i>0,\ i=1,2,\ldots,N\}$ be an open set and $\bar U$ its closure with $\bar D\subset \Delta^{N-1}$. Let $C^r_c(\bar D)$ be the set of compactly supported $C^r$ functions whose support contains $\bar D$, and let $\Sigma^r$ be either $C^r(\T^N)/\R$ or  $C_{c}^r(T^*\T^N)/C_{c}^r(\R^N)$.  Let $\cS^r$ be the unit sphere in $\Sigma^r$. 
\begin{Thm}\label{Thm1}
		There exists $r_0>0$ such that for all $r>r_0,$ the following holds. There is a residual set $\cR\subset C_c^r(\bar D)$ such that for all $\omega\in \R^N_+$ and each $\langle f\rangle\in \cR$, there is a sequence of intervals $(a_n,b_n)\subset (0,1)$ with $b_n\to 0$ and   asymptotically residual sets  $\mathcal R_n\subset \cS^r$, such that for each $\eps_n\in (a_n,b_n)$ and each $[f]\in \mathcal R_n$, the Hamiltonian system $H_{\eps_n}$ in \eqref{EqHamMain} admits at least $2^{N-1}$ distinct periodic orbits  $\gamma_{\eps_n,i},\ i=1,2,\ldots,2^{N-1},$ all intersecting the support of $f$ on the energy level 1. Moreover, if $\omega$ is nonresonant $($i.e. there is no integer vector $\bk\neq 0$ such that $\bk\cdot\omega=0)$, then for any sequence $\eps_n\in(a_n,b_n)$, the union of all periodic orbits $\cup_{i,n}\gamma_{\eps_n,i}$  is dense, when projected to $\T^N$. 
\end{Thm}
Note that we do not require $f$ to be nonnegative or nonpositive as in the definition of strong closing property. 
To appreciate better the role played by the asymptotically residual condition, we recall the well-known example of Katok. 
 In \cite{Ka} (see also \cite{Zi}), Katok constructed a family of Finsler metrics $F_\eps$ on the sphere $\mathbb S^N$, which reduces to the standard metric when $\eps=0$. For irrational $\eps$, the metric admits only $N$ (for $N$ even) or $N+1$ (for $N$ odd) closed geodesics. In our theorem, if supp$f\cap\partial( \Delta^{N-1}\times \T^N)=\emptyset$, then the system $H_\eps$ admits $N$ trivial periodic orbits given by $$\{I_i=x_i^2+y_i^2=a_i/\pi, I_j=x_j=y_j=0,\ j\neq i\},\ \ i=1,2,\ldots,N,$$ and the periodic orbits given by the last theorem are nontrivial. Katok's example is excluded by our asymptotically residual condition, since it has no interval of parameters admitting nontrivial periodic orbits, but is still included in the strong closing lemma of \cite{CDPT,CS}.  
 
Furthermore, we can indeed get asymptotically denseness on the energy level, if we get rid of a zero measure set of $\omega. $ 
\begin{Thm}\label{ThmMain}
		There exists $r_0>0$ such that for all $r>r_0,$ the following holds. 
	There is a full measure set $\Omega\subset \R_+^N$ and a residual set $\cR\subset C_c^r(\bar D)$, such that for each $\omega\in \Omega$ and each $\langle f\rangle\in \cR$, there are a sequence  of intervals $(a_n,b_n)\subset (0,1)$ with $b_n\to 0$ and a sequence of   asymptotically residual sets  $\mathcal R_n\subset \mathcal S^r$, such that for each  $[f]\in \mathcal \cR_n$ and each $\eps_n\in (a_n,b_n)$, the Hamiltonian systems $H_{\eps_n}$ in \eqref{EqHamMain} admits at least $2^{N-1}$ distinct periodic orbits  $\gamma_{\eps_n,i},\ i=1,2,\ldots,2^{N-1},$ on energy level 1. Moreover, the closure of the union   $\cup_{n}\cup_i\gamma_{\eps_n,i}$ contains $D\times \T^N$. 
\end{Thm}
Lying in the heart of the proof is a recent probabilistic number theoretic result on Diophantine approximation proved by Shapira-Weiss \cite{SW}, in addition to the sharp normal form developed in the proof of Theorem \ref{ThmMain}. 
The  conclusion part of the last theorem gives more information than the strong closing property, and indeed, the proof is constructive and effective. We know quite clearly the locations of each periodic orbit and its period information.

\subsection{Speculations}
Let us compare the following three approaches for the problem of closing lemma: 
\begin{enumerate}
	\item[(a)] closing nonwandering points, 
	\item[(b)] spectral approach, 
	\item[(c)] normal form approach. 
\end{enumerate}
Each has strengthes and weaknesses. The approach (a) gives dense periodic orbits but is only $C^2$ generic (for Hamiltonian). The approach (b) gives  $C^\infty$ closing lemma for Reeb flows, but only for 3D and surface Hamiltonian diffeomorphisms. The approach (c) works for any dimension and any high regularity, but requires asymptotic genericity and gives asymptotic denseness and only for nearly integrable systems. 

Our asymptotically residual set and asymptotic denseness look a bit artificial and unsatisfactory, but it seems to us that they are intrinsic to our approach. Indeed, as we will see in the proofs of the above theorems, a nonperturbative subsystem emerges after application of the normal form. The problem of finding periodic orbit in the original system is turned into a problem of finding nondegenerate critical points in the nonperturbative subsystem. Thus we have gone beyond the setting of nearly integrable systems considered by Poincar\'e.

\subsection{Further motivations from physics}


Hamiltonian systems \eqref{EqHam}  with $N=2$ were studied intensively in both mathematics and physics literature. It has important applications in physics such as Frenkel-Kontorova model in solid physics. The interests in the $N>2$ cases, in particular the system \eqref{EqHamMain} modeling an interacting chain of driven linear rotors, arise in the recent researches as the classical model for the quantum many-body localization (c.f. \cite{KGRG,RG} etc), which is a phenomenon of localization of eigenstates and breaking of ergodicity in interacting many-body system. The classical origin of the localization is attributed to the KAM behavior when $\omega$ satisfies Diophantine condition and when there is a lack of KAM tori, delocalization may occur. Our work reveals the richness of dynamics, i.e. the nonintegrability,  of the system \eqref{EqHamMain},  even when $\omega$ is Diophantine. Thus, models related to \eqref{EqHamMain} clearly deserve further investigations with interests beyond the existence of periodic orbits.  We expect that there exist lots of KAM tori around each elliptic periodic orbit with Floquet multipliers Diophantine and wiskered tori for each hyperbolic-elliptic type periodic orbit whose elliptic part has Diophantine Floquet multipliers, and they should also play a role in the localization. 


\subsection{Organization of the paper}
The paper is organized as follows.  We give an outline of the proof in Section \ref{SSOutline}. 
In Section \ref{SFrame}, we present the framework that we use to prove all the above theorems. 
In Section \ref{SEllipsoid}, we study the case of ellipsoid and give the proof of Theorem \ref{Thm1} and \ref{ThmMain}.
In Section \ref{STorus}, we study the torus case of ellipsoid and give the proof of Theorem \ref{ThmTorus} and \ref{ThmRm}.
In Section \ref{SNF}, we give the proof of the KAM normal form. 
Finally, we have two appendices. In Appendix \ref{AppAbraham}, we include Abraham's transversality theorem. In Appendix \ref{SRelation}, we show how our result is related to the original conjecture of Irie, by relating the Hamiltonian dynamics to Reeb dynamics.

\subsection{Outline of the proof}\label{SSOutline}
The proof consists of a hard part and a soft part. The hard part is a KAM normal form as well as its associated framework for studying similar problems, and the soft part is the genericity argument as well as some results on Diophantine approximation. We present the normal form package in Section \ref{SFrame}, which is expected to have more applications. On the other hand, the soft part is usually nonstandard and requires special observations for different problems. 

We first invoke the following classical Dirichlet's simultaneous Diophantine approximation theorem. 
\begin{Lm}[Dirichlet]
	For any $\omega\in \R^N$, and any $T (>1)\in \R$, there exist $q\in \N$ and $\mathbf p\in \Z^N$, such that $q\leq T$ and 
	$$|\omega-\mathbf p/q|\leq \frac{1}{|q|T^{1/N}}.$$
\end{Lm}

The lemma motivates the following definition. 
\begin{Def}
An integer vector $(\mathbf p,q)\in \Z^N\times \N$ is called a \emph{ best approximant} to $\omega$ if for any $q'<q$ and any $\mathbf p'\in \Z^N,$ we have $\|q\omega-\mathbf p\|<\|q'\omega-\mathbf p'\|,$ where $\|\cdot\|$ is the Euclidean norm. 
	
\end{Def}

When $\omega$ is proportional to an integer $\bp\in \Z^N$, we shall take $\omega=\bp/q$, with $q=(\|\bp\|/\|\omega\|)$. In this case, we can choose $\eps$ independent of $\bp/q$.  Otherwise, we can find a sequence of best approximants $\bp_n/q_n$ by Dirichlet theorem such that $|\omega-\mathbf p_n/q_n|\leq \frac{1}{|q_n|^{1+\frac{1}{N}}}$ and choose $\eps_n$ to be a number comparable to $\frac{1}{|q_n|^{1+\frac{1}{N}}}$ whose precise values would be determined later. In the latter case, we pick $\mathbf p_n/q_n$ and $\eps_n$ and suppress the subscript $n$ for simplicity. 

With these, we shall rewrite the Hamiltonian as
\begin{equation}\label{EqHamDecomp}
	H(I,\theta)=\frac{\mathbf p}{q}\cdot I+\eps g(I)+\eps P f(I,\theta)+\eps Q f(I,\theta),
	\end{equation}
where $g(I)=\frac{1}{\eps}(\omega-\frac{\mathbf p}{q})\cdot I+\langle f\rangle (I)$, $\langle f\rangle (I)=\int_{\T^N}f(I,\theta)d\theta$, and $P f$ (respectively $Q f$) consists of all the Fourier modes with $\bk\cdot\bp=0$ (respectively $\bk\cdot \bp\neq 0$), $\bk\in \Z^N\setminus\{0\}$.  We have the freedom to choose an approximating rational vector $\bp/q$ and $\eps$. 

The next step is to obtain a KAM normal form that is to find a sympletic transformation to write the Hamiltonian into the form 
$$H(I,\theta)=\frac{\mathbf p}{q}\cdot I+\eps G(I)+\eps\bar F(I,\theta)+\eps \tilde F(I,\theta),$$
where $G(I)$ is a small perturbation of $g$, the term $\bar F$ is a small perturbation of $Pf$ still consists of Fourier modes with $\bk\cdot \bp=0,\bk\neq 0$, and $\eps F$ is negligible compared with other terms. See Proposition \ref{PropNF}. 

Note that the Hamiltonian $\frac{\mathbf p}{q}\cdot I$ determines a linear flow on $\T^N\times \{I\}$ with velocity $\frac{\mathbf p}{q}$, thus each orbit is periodic. If we take the terms $\eps G(I)+\eps\bar F(I,\theta)$ into consideration, restrict to an energy level and pick a Poincar\'e section, a critical point of $\eps G(I)+\eps\bar F(I,\theta)$ on which would give a periodic orbit of the system $\frac{\mathbf p}{q}\cdot I+\eps G(I)+\eps\bar F(I,\theta)$ that is also a linear flow with velocity $\frac{\mathbf p}{q}$. If we further assume that the critical point is nondegenerate and the eigenvalues of the Hessian are bounded away from zero by a number much larger than $\eps\|\tilde F\|_{C^2}$, then by implicit function theorem adding the perturbation $\eps \tilde F(I,\theta)$ does not spoil the periodic orbit. This is why we have the asymptotically residual condition. 

We remark that since we turn the problem of finding periodic orbit into a problem of finding nondegenerate critical points, we expect that this is a place where the normal form theory can meet the Floer theory to yield potentially more results. 


We have outlined why we have Theorem \ref{Thm1}. The periodic orbit obtained in this way is a small perturbation of the linear flow with velocity $\bp/q$, thus in the limit $\eps\to0$, would be dense when projected to $\T^N$ if $\omega$ is non resonant. However, the $I$-component is almost a constant along the orbit. To find dense periodic orbits also in the $I$-component, we note that the $I$-component of the orbit described above is very close to the critical point of $g(I)=\frac{1}{\eps}(\omega-\frac{\mathbf p}{q})\cdot I+\langle f\rangle (I)$, i.e. solution of the equation $\frac{1}{\eps}(\omega-\frac{\mathbf p}{q})+D\langle f\rangle (I)=0$. To realize a dense set of $I$ as solutions to the last equation, we invoke the following recent result of Shapira-Weiss. 

Let $\mathrm{disp}(\omega,\mathbf p,q):=q^{1/N}(\mathbf p-q\omega)\in \R^N$ be the \emph{displacement}. 

\begin{Thm}[Theorem 1.1 of \cite{SW}]\label{ThmSW}

There exists a measure $\mu$ on $\R^N$ such that for Lebesgue almost every $\omega\in \R^N$, the sequence $\mathrm{disp}(\omega,\mathbf p_k,q_k)$ equidistributes on $\R^N$ with respect to $\mu$ as $k\to\infty$, where $(\mathbf p_k,q_k)$ is the sequence of best approximants. Moreover, the measure $\mu$ has bounded support, absolutely continuous with respect to Lebesgue and $\mathrm{SO}(N)$-invariant. 
\end{Thm}
By the $\mathrm{SO}(N)$-invariance, we see that the sequence $\frac{\mathrm{disp}(\omega,\mathbf p_k,q_k)}{\|\mathrm{disp}(\omega,\mathbf p_k,q_k)\|}$ equidistributes on $\mathbb S^{N-1}$ with respect to Haar, and in particular, it is dense.  

With this, we immediately have the following Lemma. 


\begin{Lm}\label{LmCritical}
For any $I_*$ where $D\langle f\rangle(I_*)$ is nonzero, there exist a subsequence $n_k\to\infty$ and a sequence $\eps_k\to 0$ such that $\frac{1}{\eps_k}(\omega-\frac{\mathbf p_{n_k}}{q_{n_k}})\to -D\langle f\rangle(I_*)$ as $k\to\infty$. 
\end{Lm}

We also need the nondegeneracy of $D^2g(I_*)=D^2\langle f\rangle(I_*)$, with which, we can find by the implicit function theorem a sequence of $I_k\to I_*$ with $\frac{1}{\eps_k}(\omega-\frac{\mathbf p_{n_k}}{q_{n_k}})+D\langle f\rangle(I_k)=0.$ We have the following result. 

\begin{Prop}\label{PropGeneric}
	For $C^r,\ r\ge 3, $ generic function $g:\ B_1\to \R$, the set $\{\det D^2g(z)=0\}$ is a finite union of submanifolds of codimension 1, where $B_1$ is the unit ball in $\R^N$. 
\end{Prop}
Putting all these ingredients together, we get the proof of Theorem \ref{ThmMain}. 

Theorem \ref{ThmRm} and \ref{ThmTorus} turn out to be similar and much simpler since the unperturbed part of the Hamiltonian can provide the needed nondegeneracy.  
\section*{Acknowledgement}
The author would like to thank  Wenmin Gong for calling his attention the paper \cite{CDPT}, Yitwah Cheung for pointing out the paper \cite{SW}, Professor Yiming Long for suggesting the reference \cite{Ka},  and Kei Irie, Shira Tanny, Jianlu Zhang, Zhengyi Zhou from illuminating discussions. The author is supported by grant NSFC (Significant project No.11790273) in China and the Xiaomi endowed professorship of Tsinghua University. 




\section{The general framework }\label{SFrame}

In this section, we present a general framework that can be used to study similar problems. 
\begin{enumerate}
	\item \emph{KAM normal form}: The first key ingredient in the framework  is the KAM normal form, which shows that after a symplectic transformation, the Hamiltonian depends on dynamics in the $\bp$ direction very weakly. 
	\item  \emph{Linear symplectic transformation}: The second step is to separate the $\bp$ direction and the directions transverse to $\bp$, for which purpose, we perform a linear symplectic transformation. 
	\item  \emph{Energetic reduction and locating a periodic orbit}: The third step is to locate a critical point in the transverse direction, that would be a periodic orbit in the truncated system. For this step, it would be helpful to perform an energetic reduction by restricting to an energy level and reduce the system to a non autonomous one.
	\item  \emph{Lyapunov center theorem}:  The last step is to introduce the perturbation by  applying Lyapunov center theorem. For the last step, we shall use the genericity assumption. 
\end{enumerate}

\subsection{The KAM normal form}
We next perform a decomposition $f=\langle f\rangle+P f+Q f$, where supposing $f(I,\theta)=\sum_{\bk\in \Z^N} f_\bk(I)e^{i\bk\cdot \theta}$, then we have $P$ the projection to Fourier modes with $\bk\cdot\bp=0,\ \bk\neq 0$ and $Qf=f-Pf-\langle f\rangle$, i.e. 
\begin{equation}\label{EqbarF}
	\langle f\rangle=\int_{\T^N}f(I,\theta)d\theta=f_{\mathbf 0},\quad P f(I, \theta)=\sum_{\bk\cdot \mathbf p=0,\bk\neq 0} f_\bk(I)e^{i\bk\cdot \theta},\quad Q f(I, \theta)=\sum_{\bk\cdot \mathbf p\neq0} f_\bk(I)e^{ i\bk\cdot \theta}. 
\end{equation}
With Fourier expansion, we define the $C^r$-norm as $\|f\|_{C^r}=\sum_{\bk} (|\bk|^r+1)|f_\bk|_{C^r}$ where $|f_\bk|_{C^r}$ is the usual $C^r$-norm of $f_\bk$  as a function of  $I$. 

For later use, we define 
\begin{equation}\label{Eqk}
\bk_i=(p_i,0,\ldots,0,-p_1,0,\ldots,0)/(g.c.d.(p_1,p_i)),	
\end{equation}
 where $-p_1$ appears in the $i$-th entry, $i=2,3,\ldots,N.$ Then for each $\bk$ with $\bk\cdot\bp=0$, we have $\bk\in \mathrm{span}\{\bk_2,\ldots,\bk_N\}$. 



\begin{Prop}[The KAM Normal form]\label{PropNF} 
	Let $H$ be the Hamiltonian \eqref{EqHamDecomp}. Then there exists $r_0$ sufficiently large such that for all $r>r_0$ and   $\|f\|_{C^r}\leq 1$, the following holds. Let $I_\star$ be a point where $D g(I_\star)$ is vanishing and $D^2 g(I_\star)$ is nondegenerate, then there exists $\eps_0$ such that for all $0<\eps<\eps_0$, there exist $I_*$ satisfying $|I_*-I_\star|\leq C\eps q$ and  a symplectic transformation $\phi$ defined on $\ B_{\gamma^{1/3}}(I_*)\times \T^N$ and being  $\eps q$-close to identity in the $C^{2}$ norm,  such that we have
	\begin{equation}\label{EqNF}
	H\circ\phi(I,\theta)=\frac{\bp}{q}\cdot I+\eps G(I)+\eps \bar F(I,\theta)+\eps \tilde F(I,\theta), 		
	\end{equation}
	where  we have 
	\begin{enumerate}
		\item $G(I)=g(I)+o(\eps q)$ in $C^2$ satisfying $DG(I_*)=0$, 
		\item $\bar F=Pf+O(\eps^{\frac{1}{2N}}\gamma) $ in $C^2$ and
		\item  $\|\tilde F\|_{C^{2}}\leq \gamma^{1+\beta}$,
	\end{enumerate}
	 where we define $\gamma:= \eps^{(r-2)\frac{N}{N+1}}$ and $\beta>0$ is a constant depending only on $N$.
\end{Prop}

We postpone the proof of the normal form to Section \ref{SNF}.
\begin{Rk}
	The normal form looks similar to  the normal form used by Lochak in his proof of Nekhoroshev theorem $($c.f. \cite{L}$)$. However, it is much more delicate and essentially different. Indeed, its proof involves the KAM scheme with superexponential decay of the remainder.  The main reason is that $PF$ and $Pf$ consist of Fourier modes with $\bk\cdot \bp=0$ and $|\bp|\sim q$, and such a resonance only occurs first when $|\bk|$ is comparable with $q$. Then by the decay of Fourier modes, we have $\|Pf\|_{C^2}\leq \frac{C}{q^{r-2}}\|Pf\|_{C^r}\sim\gamma$, which is much smaller than $\|Qf\|_{C^2}$ which is of order 1.  Thus in the proof, we not only need to make the remainder $\tilde F$ small, but also control carefully such that $G$ and $\bar F$ are dominated by $g$ and $Pf$ respectively, rather than get spoiled by $Qf\circ\phi$. 
\end{Rk}
\subsection{Linear symplectic transformation}
To reveal the information of periodic orbit clearly, we perform the following linear symplectic transformation. The point of the transformation is to separate the direction of $\mathbf p$ and that transverse to $\mathbf p$. 

Let $M$ be the matrix $\left[\begin{array}{ccc}
	\ \frac{1}{q}\mathbf p&\ \\
	0_{(N-1)\times 1}&\mathrm{Id}_{N-1}
\end{array}
\right]$ where $\frac{1}{q}\mathbf p$ is the first row, then we have $M^{-T}=\left[\begin{array}{ccc}
	\frac{q}{p_1}&0 \\
	-\frac{1}{p_1}\bar{\mathbf p}^T&\mathrm{Id}_{N-1}
\end{array}
\right]$ where $\bar{\mathbf p}=(p_2,\ldots, p_N)\in \Z^{N-1}$. Note in particular that the $i$-th row is a multiple of $\bk_i,\ i\ge 2. $ The norms of the matrices are both order 1 as $\eps\to 0. $ In the limit $\bp/q\to\omega$, we have $M\to M_\infty$ where the first row in $M$ is replaced by $\omega$ in $M_\infty. $

With this, we introduce a non standard torus $\T_{M}^N:=\R^N/(M^{-T}\Z^N)$ and a symplectic transformation 
$$\Phi_M:\ T^*\T^N\to T^*\T_{M}^N,\quad (I,\theta)\mapsto (MI,M^{-T}\theta):=(K,\vartheta). $$
Both $\Phi_M$ and $\Phi^{-1}_M$ have uniformly bounded norms as $\bp/q\to \omega.$

Explicitly, we have 
$$K_1=\frac{1}{q}\mathbf p\cdot I,\ K_i=I_i, \quad \vartheta_1=\frac{q}{p_1}\theta_1,\ \vartheta_i=-\frac{p_i}{p_1}\theta_1+\theta_i,\quad \ i=2,3,\ldots, n.$$
We next introduce the abbreviation $K=(K_1,\bar K)$ and $\vartheta=(\vartheta_1,\bar\vartheta)$, so $\bar\vartheta=-\frac{\bar{\mathbf p}}{p_1}\theta_1+\bar \theta$ and $\bar K=\bar I$. 
\begin{Lm}
	After the linear symplectic transformation, the Hamiltonian \eqref{EqHamMain} becomes 
	$$H\circ \Phi_M^{-1}(K,\vartheta) =K_1+\eps \mathsf G(K)+\eps \bar{\mathsf F}(K_1, \bar K, \bar\vartheta) +\eps \tilde {\mathsf F}(K_1,\vartheta_1, \bar K, \bar\vartheta),$$
	where $\mathsf G=G\circ M^{-1}$, $\bar{\mathsf F}=\bar F\circ \Phi_M^{-1}$ and $ \tilde{\mathsf F}=\tilde F\circ \Phi_M^{-1}$.
\end{Lm}
\subsection{Energetic reduction}
We next perform a step of energetic reduction to get rid of the dependence on $K_1$ in the perturbation. The procedure is to restrict to the energy level set $H=1$ and solve for  $K_1:=-\cK(\xi,\eta,\vartheta_1)$ to obtain the following by the implicit function theorem
\begin{equation}
	\mathcal K(\bar K,\bar\vartheta,\vartheta_1)=-1+\eps\cG(\bar K)+\eps \bar\cF(\bar K,\bar\vartheta)+\eps  \tilde\cF(\bar K,\bar\vartheta ,\vartheta_1), \quad \vartheta_1\in \R/(q\Z), 	
\end{equation}
where $\bar\cF$ and $\tilde \cF$ has similar $C^r$ norm as $\bar F$ and $\tilde F$, in particular, we have $\bar\cF(\bar K,\bar\vartheta)=\bar F(1,\bar K,\bar\vartheta)+O(\eps\gamma)$.  This is a standard procedure which turns the autonomous system $H$ into a non autonomous system $\cK$ and treats $\vartheta_1$ as the new time variable.  
\begin{Lm}
	The nonautonomous Hamiltonian system $\mathcal K$ defines the same flow as the autonomous system $H$ restricted to the energy level 1. 
\end{Lm}
\begin{proof}
	From the Hamiltonian $H$, we get the equations of motion 
	$$\begin{cases}
		\frac{dK_1}{dt}=-\frac{\partial H}{\partial \vartheta_1},\\
		\frac{d\vartheta_1}{dt}=\frac{\partial H}{\partial K_1}=1+\eps\frac{\partial \tilde{F}}{\partial K_1},\\
	\end{cases} \begin{cases}
		\frac{d\bar K}{dt}=-\frac{\partial H}{\partial \bar\vartheta},\\
		\frac{d\bar\vartheta}{dt}=\frac{\partial H}{\partial \bar K}.\\	
	\end{cases}$$
	Since in the system $\cK$, we use $\vartheta_1$ as the new time, we divide the $\frac{d\bar K}{dt}$ and $\frac{d\bar\vartheta}{dt}$ equations by the $\frac{d\vartheta_1}{dt}$ equation to get the following by implicit function theorem
	$$\begin{cases}
		\frac{d\bar K}{d\vartheta_1}=-\frac{\partial H}{\partial \bar\vartheta}/\frac{\partial H}{\partial K_1}=-\frac{\partial \cK}{\partial \bar\vartheta},\\
		\frac{d\bar\vartheta}{d\vartheta_1}=\frac{\partial H}{\partial \bar K}/\frac{\partial H}{\partial K_1}=\frac{\partial \cK}{\partial \bar K},\\	
	\end{cases}$$
	which has the form of the canonical equations. 	
\end{proof}

\subsection{Locating a periodic orbit in the truncated system}

We next show that a critical point  of $\cG+\bar \cF$ gives rise to an orbit of the system $\cK$, if we ignore the $\eps\tilde \cF$-perturbation. 
\begin{Lm}
	Let $(\bar K_*,\bar\vartheta_*)$ be a nondegenerate critical point of $\cG+\bar \cF$, i.e. $D(\cG+\bar \cF)(\bar K_*,\bar\vartheta_*)=0$ and $D^2(\cG+\bar \cF)(\bar K_*,\bar\vartheta_*)$ is nondegenerate. 	Then it is a nondegenerate fixed point of the truncated system $\frac{d}{d\vartheta_1}(\bar K, \bar\vartheta)=JD(\cG+\bar \cF)(\bar K, \bar\vartheta)$. 
\end{Lm}

\subsection{Rescaling }

Without loss of generality, we assume the critical point of $\cG+\bar\cF$ is $(\bar K_*,\bar\vartheta_*)=(0,0)$. 
We perform the Taylor expansion around it
$$\cG(\bar K)+\bar \cF(\bar K,\bar\vartheta)=\langle D^2(\cG +\bar\cF)(\bar K,\bar\vartheta),(\bar K,\bar\vartheta)\rangle+ O(|\bar K|^3+\gamma (|\bar K\bar\vartheta^2|+ |\bar K^2\bar\vartheta|+|\bar\vartheta^3|)).$$

Note that $\partial_{\bar\vartheta}^2\bar \cF$ is of order $\gamma$ much smaller than $D^2\cG$.  We introduce a rescaling
\begin{equation}\label{EqRescaling}
\bar K\mapsto \bar K/\gamma^{1/2},\quad \bar\vartheta\mapsto \bar\vartheta,\quad \cK\mapsto \cK/(\eps\gamma ),\quad \vartheta_1\mapsto \vartheta_1 \gamma^{1/2}\eps.	
\end{equation}
 This transformation does not change the form of Hamiltonian equations. 

\begin{Prop}\label{PropRescaling}
 After this transformation \eqref{EqRescaling}, the Hamiltonian becomes
$$\cK( \bar K,  \bar\vartheta,\vartheta_1)= \langle D^2\cG|_{(0,0)} \bar K,\bar K\rangle+\frac{1}{\gamma}\langle\partial_{\bar\vartheta}^2\bar \cF|_{(0,0)}\bar\vartheta,\bar\vartheta\rangle+O(\gamma^{1/2})+\frac{1}{\gamma} \tilde \cF(\gamma \bar K,  \bar\vartheta,\eps^{-1}\gamma^{-1/2}\vartheta_1) $$
in the region $(\bar K,\bar\vartheta)\in B_1(0)$, where we have discarded a constant $1/(\eps\gamma)$. 
\end{Prop}
\begin{proof}
	The $\frac{1}{\gamma}\tilde \cF$ term is straightforward. The term $\eps\langle D^2\cG\bar K,\bar K\rangle$ is mapped to $\eps\gamma\langle D^2\cG\bar K,\bar K\rangle$ under the transformation $\bar K\mapsto \bar K/\gamma^{1/2}$. The factor $\eps\gamma$ gets canceled after dividing the Hamiltonian by $\eps\gamma$. Similarly, the term $\langle\partial_{\bar\vartheta}^2\bar \cF\bar\vartheta,\bar\vartheta\rangle$ gets a factor $1/\gamma$ after the whole procedure and the term $\partial_{\bar\vartheta}\partial_{\bar K}\bar \cF\bar K\bar\vartheta$ gets a factor $\gamma^{1/2}$. The latter is of order $\gamma^{1/2}$ since $\|\bar \cF\|_{C^2}\sim\gamma$. For the same reason, the term $O(|\bar K|^3+\gamma (|\bar K\bar\vartheta^2|+ |\bar K^2\bar\vartheta|+|\bar\vartheta^3|))$ is of order $\gamma^{1/2}$. 
\end{proof}

\subsection{Lyapunov center theorem}
We cite the following Lyapunov center theorem. 
\begin{Thm}[Proposition 9.1.1 of \cite{MHO}]\label{ThmLyapunov}
	If the nontrivial multipliers of a periodic orbit of a system are not 1, then the periodic orbit can be continued. 
\end{Thm}
In our case, the theorem can be understood easily as follows. Suppose we have a map $\phi:\ B_1\to B_1$ with a fixed point at 0 and $D\phi(0)$ has no eigenvalue 1, where $B_1$ is the unit ball in $\R^{2N}$. Then a small $C^1$ perturbation of $\phi$ also admits a fixed point by implicit function theorem. 

We next prove a general proposition guaranteeing the existence of periodic orbit when a normal form and nondegeneracy condition is given.  
\begin{Prop}\label{PropFinal}
	Suppose the eigenvalues of the matrix $e^{J\Xi}$ are bounded away from 1 by a number much larger than $\gamma^{\beta}$, where $\Xi=\left[\begin{array}{cc}
D^2\cG&0\\
0&\gamma^{-1}\partial^2_{\bar\vartheta\bar\vartheta}\bar\cF	\end{array}\right]\Big|_{(0,0)}$, then the system in Proposition \ref{PropRescaling} admits a periodic orbit. This implies that the original system $H_\eps=\omega\cdot I+\eps f$, if $\omega$ is non resonant, admits a periodic orbit that intersects any ball of radius greater than $2/q$, when projected to $\T^N. $
\end{Prop}
\begin{proof}
The Hamiltonian equation of the system $\cK$ in Proposition \ref{PropRescaling} is written as 
$$\frac{d}{d\vartheta_1}\left[\begin{array}{c}
	\bar K\\
	\bar \vartheta
\end{array}\right]=J\Xi \left[\begin{array}{c}
\bar K\\
\bar \vartheta
\end{array}\right]+ O(\gamma^{\beta}),$$
where the $O$ is in $C^1$-norm and we do not take derivative with respect to $\vartheta_1$ on the RHS. 

Note that the $\vartheta_1$ variable is defined on the circle $\R/(\eps\gamma^{1/2} q\Z)$, since $\vartheta_1$ before the rescaling \eqref{EqRescaling} is defined on $\R/(q\Z)$. We take $T$ to be $ [(\eps\gamma^{1/2} q)^{-1}](\eps\gamma^{1/2} q)$ that is a number close 1, by repeating on the circle $\R/(\eps\gamma^{1/2} q\Z)$ for $[(\eps\gamma^{1/2} q)^{-1}]$ times. 

	
	To apply the Lyapunov center theorem, it is enough to show that the $O(\gamma^\beta)$-perturbation gives a $C^1$-small perturbation to the time-$T$ map of the Hamiltonian system. The differential of the time-$T$ map is obtained by integrating over time  $T$ the variational equation $$\frac{d}{d\tau}\left[\begin{array}{c}
	\Delta \bar K\\
	 \Delta\bar \vartheta
	\end{array}\right]=(JD^2\cK(\bar K, \bar \vartheta,\vartheta_1))\left[\begin{array}{c}
	\Delta \bar K\\
	\Delta\bar \vartheta
\end{array}\right],$$ which is given by differentiating the Hamiltonian equation $\frac{d}{d\tau}(\bar K, \bar \vartheta)=JD\cK(\bar K, \bar \vartheta,\vartheta_1)$.  Note that in these equations there is no derivative with respective to $\vartheta_1$.   By DuHamel principle, the solution to the variational equation is given by $$\left[\begin{array}{c}
\Delta \bar K\\
\Delta\bar \vartheta
\end{array}\right](T)=e^{J\Xi T}\left[\begin{array}{c}
\Delta \bar K\\
\Delta\bar \vartheta
\end{array}\right](0)+\int_0^T e^{J\Xi (T-s)}(O(\gamma^{\beta}))\left[\begin{array}{c}
\Delta \bar K\\
\Delta\bar \vartheta
\end{array}\right](s)ds,$$ thus the differential of the time-$T$  map  is a $\gamma^\beta$ perturbation of $e^{J\Xi T}$. 
	We can then apply Lyapunov center theorem to conclude that the perturbed Hamiltonian system $\cK$ also admits a periodic orbit. 
	The periodic orbit passes through a sufficiently small neighborhood of the maximum point of $f$ thus passes through supp$f$. 
	
	The periodic orbit is $O(\gamma^\beta)$-close to the unperturbed one $\left[\begin{array}{c}
		\Delta \bar K\\
		\Delta\bar \vartheta
	\end{array}\right]=0$. Going back to the original system by undoing $\Phi_M$, we see that this periodic orbit is an $O(\gamma^\beta)$ perturbation of a periodic orbit with slope $\frac{\mathbf p}{q}$ when projected to $\T^n$. Since we have $\gamma= q^{-(r-2)\beta}\ll \frac{1}{q}$ and $\frac{\mathbf p_n}{q_n}\to \omega$, we get complete the proof of the statement. 
\end{proof}

\section{The ellipsoid case}\label{SEllipsoid}
In this section, we give the proofs of Theorem \ref{Thm1} and \ref{ThmMain}. By the above framework, it is  enough to verify the assumptions of Proposition \ref{PropFinal}.

\subsection{Proof of Theorem \ref{Thm1}}
In this section, we give the proof of Theorem \ref{Thm1}.

\begin{proof}[Proof of Theorem \ref{Thm1}]

We choose an interval $(a,b)\subset(0,1)$ with $b$ sufficiently small and introduce a sequence of intervals $(a_n,b_n):=(a,b)q_n^{-1-\frac{1}{N}}$, where we consider the sequence of best approximants $\bp_n/q_n$. We shall choose any sequence $\eps_n\in (a_n,b_n). $ We choose $b$ so small that we have $\frac{1}{\eps_n}|\omega-\frac{\bp_n}{q_n}|\le b\ll\max|\langle f\rangle|$. Suppose $I_\star$ is a maximum point of $\langle f\rangle|_{\Delta^{N-1}}$ where $D^2 \langle f\rangle$ is nondegenerate. With this, we rewrite the Hamiltonian as \eqref{EqHamDecomp} and get that  $g$ has a maximum point $O(b)$-close to $I_\star$. 

We next apply the normal form Proposition \ref{PropNF} to it. To verify the assumptions of Proposition \ref{PropFinal}, we need the following nondegeneracy conditions. 

We pick a sequence of $c_n\to 0+$ satisfying $c_n\gg\eps_n^{\frac{1}{2N}}$, for instance, we can take $c_n=\eps_n ^{\frac{1}{4N}}. $ We next introduce $\cR_n$ to be the set of functions $f\in C^r(\bar U\times \T^N)$ satisfying
\begin{enumerate}
	\item The absolute values of the eigenvalues of $D^2\langle f\rangle(I_\star)$ are bounded away from 0 by $c_n$;
	\item The absolute values of the  eigenvalues of $\partial^2_{\bar\vartheta\bar\vartheta}(Pf\circ\Phi^{-1}_M)(K_\star,\bar\vartheta_\star)$ are bounded away from zero by $c_n q_n^{r-2}$, where $K_\star=MI_\star$ and $\bar\vartheta_\star$ is a critical point of $Pf\circ\Phi^{-1}_M(K_\star,\cdot).$
\end{enumerate}
To describe item (2) better, we note that $Pf$ consists of Fourier modes in $f$ satisfying $\bk\cdot\bp=0,\ \bk\neq 0$. Thus we write $Pf(I_\star,\theta)=\psi(\bk_2\cdot\theta,\ldots,\bk_{N}\cdot\theta)$, where $\bk_2,\ldots,\bk_{N}$ are defined in \eqref{Eqk} with modulus of order $q_n$. Then a nondegenerate critical point of $\psi$ where the norms of eigenvalues of Hess$\psi$ is bounded away from zero by $c_n$, corresponds a nondegenerate critical point of $Pf$ with norms of eigenvalues of the Hess bounded away from zero by $c_n/q_n^{r-2}$. By normal form Proposition \ref{PropNF}, the subleading term in $PF$ is $O(\eps^{\frac{1}{2N}}\gamma)$ in $C^2$, thus is not going to spoil the nondegeneracy of $Pf.$ The sequence of sets $\{\cR_n\}$ is a sequence of asymptotically residual sets for the same reason as the example of Morse lemma in the introduction. 

After the energetic reduction, by implicit function theorem, we see that there is a critical point $(\bar K_*,\bar \vartheta_*)$ of $\cG+\bar\cF $ that is $O(\eps_n)$ close to $\Phi_M(I_\star,\theta_\star)$ and the eigenvalues of $\Xi |_{(\bar K_*,\bar \vartheta_*)}$ (in Proposition \ref{PropFinal}) are bounded away from zero by $c_n/2$. Readers can either prove it directly or refer to  Lemma \ref{LmNonDeg} in the next subsection. The point $I_\star$ is fixed by $\langle f\rangle$ but there are at least $2^{N-1}$ $\bar\vartheta_\star$s where $Pf\circ\Phi_M^{-1}(K_\star,\bar\vartheta)$ has a critical point since $\bar\vartheta$ is defined on $\T^{N-1}$. To each of these $\bar\vartheta_\star$s, we assume the above asymptotically residual condition $(2)$. Thus by applying Proposition \ref{PropFinal}, we get $2^{N-1}$ periodic orbits. 

\end{proof}

\subsection{Proof of Theorem \ref{ThmMain}}
In this section, we give the proof of Theorem \ref{ThmMain}. From the proof of Theorem \ref{Thm1}, we see that in the $\T^N$-component, the periodic orbit would be dense in the limit $\eps_n\to0$, while in the $I$-component is almost constant. To prove the denseness also in the $I$-component as $\eps_n\to0$, we use the freedom of adjusting $\eps_n$ and Theorem \ref{ThmSW}.  

\subsubsection{Locating the critical point of $\cG+ \bar\cF$}
We first have the following lemma. 
\begin{Lm}\label{LmNonDeg}
	\begin{enumerate}
		\item 	Let $I_\star$ be a point satsfiying $\omega\cdot I_\star=1$
		and that $D\langle f\rangle (I_\star) $ is nonvanishing, and $D^2\langle f\rangle (I_\star) $ and $\Gamma(\langle f\rangle):=\partial_{\bar K\bar K}^2(\langle f\rangle \circ M^{-1}_\infty)$
		are both nondegenerate. 
		\item Let $\bp_k/q_k$ be a subsequence of the best approximants of $\omega$ such that there is a sequence $\eps_k\to0$ satisfying 
		\begin{equation}\label{EqIstar}
			\quad \frac{\omega-\bp_k/q_k}{\eps_k}+D\langle f\rangle(I_\star)\to 0	.
		\end{equation}
		\item Let  $\theta_\star$ be a point such that $\partial_{\bar\theta}P f(I_\star,\theta_\star)=0$ and $\partial^2_{\bar\theta\bar\theta}P f(I_\star,\theta_\star)$ is uniformly nondegenerate in the sense that $\frac{1}{\gamma}\partial^2_{\bar\theta\bar\theta}P f(I_\star,\theta_\star)$ is much larger than $\eps_k^{\frac{1}{2N}}$. 
	\end{enumerate}
	Then for $\eps_k$ sufficiently small, there is a critical point $(\bar K_{k},\bar\vartheta_{k})$ of the function $\cG+\bar \cF$ $($recall that $\cG+\bar \cF$ depends on $\bp_k/q_k$$)$ satisfying $$(\bar K_k,\bar\vartheta_k)\to \pi_{-1}\circ\Phi_M(I_\star,\theta_\star),\quad \mathrm{as}\ k\to\infty,$$
	where $\pi_{-1}:\ T^*\T_M^N\to T^*\T^{N-1}$ is defined as removing the entries corresponding to $K_1$ and $\vartheta_1.$
\end{Lm}
\begin{proof}
	To find critical points of $\eps \cG+\eps \bar \cF$, we solve the equation 
	$$D_{\bar K}\cG+\eps D_{\bar K}\bar \cF=0,\quad \gamma^{-1}D_{\bar \vartheta} \bar \cF=0. $$
	Note that this equation can be considered as an $O(\eps)$ small perturbation of $$D_{\bar K}\cG=0,\ \gamma^{-1}D_{\bar\vartheta} \bar \cF=0,$$ which can be easily solved, then the perturbed equation can be solved by implicit function theorem. 
	
	We thus first locate the critical point of $D_{\bar K}\cG$. We again consider the further truncated system $K_1+\eps (G\circ M^{-1})(K_1,\bar K)=1$, from which we solve for $K_1$. So by implicit function theorem, we get
	$$D_{\bar K }K_1=-\frac{\partial_{\bar K }(G\circ M^{-1})}{1+\eps \partial_{K_1}(G\circ M^{-1})}.$$
	
	Up to an error of order $\eps$, we may set $K_1=1$ in the argument of $G\circ M^{-1}$, and solve for $\bar K$ in the equation $\partial_{\bar K}(G\circ M^{-1})|_{K_1=1}=0. $ The last equation is satisfied at the point $I_\star$ where $DG=0.$ By the assumption that $D^2G(I_\star)=D^2\langle f\rangle(I_\star)+O(\eps q)$ (see Proposition \ref{PropNF}) is nondegenerate, we find that there is a critical point of $\cG$ that is $O(\eps)$-close to $M I_\star$. Note that $\|M-M_\infty\|=O(\eps)$, thus we can interpret the critical point of $\cG$ as a critical point $G$ restricted to the hyperplane $\{\omega\cdot I=1\}$ up to an $O(\eps)$-error. 
	

By the nondegeneracy of $D^2\langle f\rangle(I_\star)$ and assumption (2), we see that the function $g(I)=\frac{1}{\eps_k}(\omega-\frac{\bp_k}{q_k})\cdot I+D\langle f\rangle(I)$ has a critical point $I_k$ that converges to $I_\star$ as $k\to\infty. $ Thus correspondingly, we get a sequence $\bar K_k$ as critical points of $\cG$ converging to $\pi_{-1}MI_\star$ as $k\to\infty$. 

	We next consider the Hessian $D^2_{\bar K}\cG$ at $\bar I_\star$. Continuing the above reasoning, we get 
	\begin{equation*}
		\begin{aligned}
			&D^2_{\bar K}\cG=\Gamma(G)=\partial_{\bar K\bar K}^2(\langle f\rangle \circ M^{-1}_\infty)+O(\eps),
		\end{aligned}
	\end{equation*}
which up to an $O(\eps)$ error, can be interpretted as the Hessian of $g$ as a function over the hyperplane $\{\omega\cdot I=1\}$ at the point $I_\star$, where the role played by the matrix $M_\infty$ is to send the hyperplane $\{\omega\cdot I=1\}$ to the hyperplane $K_1=1$. 
	
	We next consider the equation $D_{\bar\vartheta} \bar \cF=0. $ By implicit function theorem, we have 
	$$\partial_{\bar\vartheta}\cK=-\frac{\partial_{\bar \vartheta }(\bar F\circ \Phi_M^{-1})}{1+\eps \partial_{K_1}(\bar F\circ\Phi_M^{-1})}.$$
	Note that we have $	 e^{ i \bk \cdot \theta}=e^{ i (k_1\theta_1+\bar \bk \cdot \bar\theta)}=e^{ i (\frac{\bk\cdot \mathbf p}{q}\vartheta_1+\bar \bk \cdot \bar\vartheta)}	$, so for terms in $\bar \cF$ with $\bk\cdot \bp=0$, we have $	 e^{ i \bk \cdot \theta}=e^{ i \bar \bk \cdot \bar\vartheta}	$ and $ \partial_{\bar \vartheta }\bar \cF=\partial_{\bar \theta }\bar F$. By the normal form Proposition \ref{PropNF}, we know the leading term in $\bar F$ is given by $Pf$. 
	
\end{proof}

\subsubsection{Proof of Theorem \ref{ThmMain} assuming Proposition \ref{PropGeneric}}

We next complete the proof of the main theorem assuming  Proposition \ref{PropGeneric}. Note that $\Gamma(\langle f\rangle)$ at $I_\star$ is the Hessian of the function $\langle f\rangle$ written as a function over the hyperplane $\{\omega\cdot I=1\}$ at the critical point $I_\star$, where $M_\infty$ plays the role of transforming $\{\omega\cdot I=1\}$ into a standard $\R^{N-1}$. 
Thus Proposition \ref{PropGeneric} shows that the Hessian $\Gamma(\langle f\rangle)$ is nondegenerate at almost every point on $\{\omega\cdot I=1\}$. 

We first choose a countable dense subset $ \mathcal I=\{I_n\}$ of points on $D\subset \{\omega\cdot I=1\}$ on which  $D^2\langle f\rangle$   and $\Gamma(\langle f\rangle)$ are nondegenerate. This is guaranteed by Proposition \ref{PropGeneric} applied to $\langle f\rangle.$ For each $I_n\in \mathcal I$, we pick a sequence of approximants $\bp_{n,k}/q_{n,k}$ and $\eps_k$ by Lemma \ref{LmNonDeg}. We can perturb $\eps_k$ slightly in an interval $(a_k,b_k)$ of length $O(\eps_k^2)$ such that \eqref{EqIstar} holds for all $\eps_k\in (a_k,b_k)$ as $k\to\infty.$ Now we relabel the sequence $(\bp_{n,k}/q_{n,k},\eps_{n,k})_{n,k\in \N}$ by $(\bp_{n}/q_{n},\eps_{n})_{n\in \N}$ such that $\eps_n$ decreases to zero. 

For each $n$, we choose $c_n$ as in the proof of Theorem \ref{Thm1} and introduce $\cR_n$. Thus we get a sequence of asymptotically residual sets. Applying Theorem \ref{Thm1} and the last lemma, for each $n$, we get a periodic orbit that is $o(1)$ close to $I_n$ as $n\to\infty$ when projected to $I$-component, and is $q_n^{-1}$-dense when projected to $\T^N$. Taking union over all $n$, we get that the periodic orbits has closure containing $\bar D\times \T^N$ as in the statement. 
\qed

\subsubsection{Proof of  Proposition \ref{PropGeneric}}

We next give the proof of Proposition \ref{PropGeneric}. This fact is a higher dimensional generalization of the basic fact that generic $C^2$ function $f$ on the interval $[0,1]$ has only finitely many points of inflection (where $f''=0$).

\begin{proof}[Proof of Proposition \ref{PropGeneric}]
	We invoke a transversality theorem of Abraham(see Theorem \ref{ThmAbraham} of Appendix).   Here we show how to set up the problem to fit into the framework of Theorem \ref{ThmAbraham}. In the following, we shall use the notations of Theorem \ref{ThmAbraham}.
	
	Let us take $A=C^r(M_1),\ r\ge 3$ including $C^\infty$, $M_1=\R^{N}$, $K\subset M_1$ a compact subset, and $M_2=\mathrm{Sym}$, where $\mathrm{Sym}$ is the space of symmetric matrices of $N\times N$, which is a manifold of dimension $N+\frac{N(N-1)}{2}$. There is a subset $V\subset \mathrm{Sym}$ consisting of singular symmetric matrices.  We next claim that 
	
	{\it  Claim: $V$ is a finite union of submanifolds in $\mathrm{Sym}$ of codimension 1.}

\noindent\emph{Proof of the claim:}
	Indeed, since each matrix $L\in \mathrm{Sym}$ admits a decomposition $L=Q^{-1}\Lambda Q$ where $Q\in \mathrm{SO}(N)$ and $\Lambda$ diagonal and real. A matrix $L\in V$ if and only if it has a zero eigenvalue, i.e. the corresponding $\Lambda$ lies in the coordinate hyperplanes in $\R^{N}$. So the space of singular diagonal matrices is the union of $N$ codimension 1 hyperplanes in $\R^{N}$. We next count the dimension. We have $dim\mathrm{Sym}=N+\frac{N(N-1)}{2}$ with $\frac{N(N-1)}{2}=dim\mathrm{SO}(N)$. This verifies the claim.  \hskip 3.84in q.e.d.

	We next consider the pseudo-representation $F:\ A\to C^1(M_1, M_2)$ via $[F(a)](z)=D^2 f(z)$, which is $C^1$ in $z$ for $a\in C^r,\ r\ge 3$. The evaluation map $ev(TF):\ A\times TM_1\to TM_2$ via $ev(TF)(a,q)=T[F(a)]q$, which is given explicitly as $ev(TF)(a,q)=D_z(D^2 a(z)) q$, which is continuous in $a$ and $q$ if $a\in C^3$. 
	
	We next verify that the  pseudo-representation $F$ is  $C^r$-pseudo-transverse to $V$ on $K$.  We take $D=C^r(K)$ and $\psi_a$ to be identity. Then we have $D$ is dense in $A$. We next show that
	
	{\it For all $a\in D$, there is a neighborhood $B_a$ of $a$ in $D$, such that 
		the evaluation map $ev(Fa):\ B_a\times M_1\to M_2$ is $C^r$ and transverse to $V$ on $a\times K$.}
	
	Indeed, without loss of generality, we assume $D^2a(z)$ is degenerate and diagonal. If $D^2a(z)$ is not diagonal, we can perform an orthogonal transformation in $T_zM_1$ to diagonalize it. Suppose $D^2a(z)=\mathrm{diag}(*,*,\ldots,*,0)$ where $*\neq 0$, then it is easy to add to $a$ a function $\eps P$ that is locally $\frac{\eps }{2}z_n^2$ such that $D^2(a+\eps P)(z)=\mathrm{diag}(*,*,\ldots,*,\eps)$.  This verifies the claim. 
	
	The result then follows by applying Theorem \ref{ThmAbraham}. 
\end{proof}

\section{The torus case}\label{STorus}
In this section, we give the proof of Theorem \ref{ThmRm} and \ref{ThmTorus}. As we shall see in the proof that there is no need for Theorem \ref{ThmSW}. Instead, we shall exploit the nondegeneracy of the unperturbed part $h$.

\subsection{The toral geodesic flow case}
In this section, we give the proof of Theorem \ref{ThmTorus}. 
We consider Riemannian metric on $\T^N$ of the form 
$$ds^2_\eps=\sum_{i,j}(\dt_{ij}+\eps a_{ij}(\theta))d\theta_id\theta_j.$$
We treat $ds^2_\eps$ as twice of the Lagrangian and perform a Legendre transformation to get Hamiltonian of the form 
$$H(\theta,y)=\frac{1}{2}\sum_{i,j}(\dt_{ij}+\eps b_{ij}(\theta))y_iy_j=\frac{1}{2}\|y\|^2+\frac{1}{2}\sum_{i,j}\eps b_{ij}(\theta)y_iy_j, $$
where the matrix $(\dt_{ij}+\eps b_{ij}(\theta))$ is the inverse of $(\dt_{ij}+\eps a_{ij}(\theta)).$ Thus solving the equation 
$$\sum_{j}(\dt_{ij}+\eps b_{ij}(\theta))(\dt_{jk}+\eps a_{jk}(\theta))=\dt_{ik},$$
we get $b_{ij}=-a_{ij}+O(\eps)$. In the following, we denote $B=(b_{ij})=-A+O(\eps)$. 

We next pick $y_*=\bp/\|\bp\|$ where $\bp\in \Z^N$, such points are dense on the unit sphere. For given $\eps$, we consider rational points with $\|\bp\|\leq \eps^{-\nu}$ for some small $0<\nu<\frac{1}{4}$. 
With this, we replace $y$ by $y_*+I$, where $I$ is small, and expand
$$H(\theta,y)=\frac{1}{2}(\|y_*\|^2+2y_*\cdot I+\|I\|^2)+\frac{1}{2}\eps( \langle B(\theta)y_*,y_*\rangle+2\langle B(\theta)y_*,I\rangle+\langle B(\theta)I,I\rangle). $$
The main difference from the ellipsoid case is that the quadratic term $\|I\|^2$ does not carry a factor $\eps$ as $\eps g$ does in system \eqref{EqHamDecomp}. If we examine the proof of Proposition \ref{PropNF}, we shall see that the reminder $Qf$ here does not decay fast enough. To solve this problem we introduce the following procedure called $\sqrt\eps$-blowup. 
\subsubsection{$\sqrt\eps$-blowup}\label{SSSBlowup}
We consider an $\sqrt\eps$-neighborhood of $I=0$ and rescale $$I\mapsto\sqrt\eps I, \quad \theta\mapsto \theta,\quad H\mapsto H/\sqrt \eps,\quad t\mapsto t\eps . $$
This transformation blows up the $\sqrt\eps$-neighborhood of $I=0$ into a neighborhood of unit size. The Poincar\'e-Cartan form gets multiplied by $\sqrt\eps$, thus the form of Hamiltonian equations does not change. 
The Hamiltonian after the blowup has the form  
$$H(I,\theta)=y_*\cdot I+\sqrt{\eps}\frac{1}{2}\|I\|^2+\frac{\sqrt{\eps}}{2}\big( \langle B(\theta)y_*,y_*\rangle+2\sqrt\eps\langle B(\theta)y_*,I\rangle+\eps\langle B(\theta)I,I\rangle\big) $$
which is defined on $B_1(0)\times \T^N$ and  we have discarded the constant $\frac{1}{2\sqrt\eps }\|y_*\|^2$. 
We next donote by $f(I,\theta)$ the term in the big parenthesis, whose leading term is $\langle B(\theta)y_*,y_*\rangle$. We next decompose $f=\langle f\rangle+Pf+Qf$ and denote by $g=\frac{1}{2}\|I\|^2+\langle f\rangle$. Without loss of generality, we discard the constant $\int_{\T^N}\langle B(\theta)y_*,y_*\rangle d\theta$ in $\langle f\rangle$, thus $\langle f\rangle=O(\sqrt\eps)$. Thus the leading term in $g$ is $\frac{I^2}{2}$ which is $\theta$-independent and  the leading term in $Pf$ is $\langle PB(\theta)y_*,y_*\rangle=-\langle PA(\theta)y_*,y_*\rangle+O(\eps)$, which is $I$-independent. 

Thus we get the form 
$$H=\frac{\bp}{\|\bp\|}\cdot I+\sqrt\eps (g(I)+ Pf+ Qf). $$
This is in the form of equation \eqref{EqHamDecomp}, with the only difference being that $\eps$ is now replaced by $\sqrt\eps$ and $q$ replaced by $\|\bp\|$. Since we choose $\|\bp\|\leq \eps^{-\nu}$ with $\nu<1/4$, we get $\sqrt\eps \|\bp\|\leq \eps^{1/4}$, which is the analogue of $\eps q$ in Proposition \ref{PropNF}.  

\subsubsection{The KAM normal form}
Analogous to Proposition \ref{PropNF}, we have 
\begin{Prop}\label{PropNFTorus}
	There exists $r_0$ sufficiently large such that for all $r>r_0$ and   $\|A\|_{C^r}\leq 1$, the following holds. There exists $\eps_0$ such that for all $0<\eps<\eps_0$, there exist $I_*$ satisfying $|I_*|\leq C\eps^{\frac{1}{2}-\nu}$ and  a symplectic transformation $\phi$ defined on $\ B_{\gamma^{1/3}}(I_*)\times \T^N$ and being  $\eps^{\frac{1}{2}-\nu}$-close to identity in the $C^{2}$ norm,   such that 
	$$H\circ\phi(I,\theta)=\frac{\bp}{\|\bp\|}\cdot I+\sqrt\eps G(I)+\eps \bar F(I,\theta)+\eps \tilde F(I,\theta),$$
	where we have $\|\tilde F\|_{C^2}\leq \gamma^{1+\beta}$, $\bar F=-\frac{1}{2}\langle PA(\theta)y_*,y_*\rangle+O(\eps^{\frac{1}{4N}}\gamma)$ in $C^2$ and $G(I)=\frac{I^2}{2}+O(\eps^{\frac{1}{2}-\nu})$ in $C^2$ and $DG(I_*)=0$. 
\end{Prop}

\subsubsection{Proof of Theorem \ref{ThmRm}}
The remaining argument is similar to that of Theorem \ref{Thm1}. After the energetic reduction, we get a Hamiltonian of the form 
$$\cK=1+\sqrt{\eps}\cG(\bar K)+\eps\bar\cF(\bar K,\bar\vartheta)+\eps\tilde \cF(\bar K,\bar\vartheta,\vartheta_1),$$
where the leading term in $\cG$ is $\frac{\bar K^2}{2}$, and the leading term in $\bar\cF$ is $-\frac{1}{2}\langle PA\circ M^{T}(\bar\vartheta)y_*,y_*\rangle$. 

The function $\cG$ has a single critical point close to $\bar K=0$ that is automatically nondegenerate. We next consider critical point of $\bar \cF$. 
A nondegenerate critical point of the function $-\frac{1}{2}\langle PA\circ M^{T}(\bar\vartheta)y_*,y_*\rangle$ corresponds to a periodic orbit in the truncated system $1+\frac{\bar K^2}{2}+\frac{1}{2}\langle PA\circ M^{T}(\bar\vartheta)y_*,y_*\rangle$. Thus we need some uniform nondegeneracy to handle the perturbations.  We relabel all the $y_*=\bp/\|\bp\|$ by $\bp_n/\|\bp_n\|$ according to non decreasing order of $\|\bp_n\|$, and let $c_n\to 0+$ be a sequence as in the proof of Theorem \ref{Thm1}. We introduce $\cR_n$ to be the space of $C^r$ functions such that the norms of eigenvalues of the Hessian of $\frac{1}{2}\langle PA\circ M^{T}(\bar\vartheta)y_*,y_*\rangle$ at critical points is bounded away from zero by $c_n  \|\bp_n\|^{-(r-2)}$. This gives us the sequence of asymptotically residual sets $\cR_n$. 

With this choice of asymptotically residual sets, we find a periodic orbit for each $y_*=\bp/\|\bp\|$ with $\|\bp\|\leq \eps^{-\nu}$ by applying Proposition \ref{PropFinal}. As $\eps\to 0$, we find a periodic orbit for all $y_*=\bp/\|\bp\|$ with $\bp\in \Z^N\setminus\{0\}$. Since the set $\{\bp/\|\bp\|, \ \bp\in \Z^N\setminus\{0\}\}$ is dense on the unit sphere, we get the statement of the theorem. 

We can make the estimate quantitative as follows. Let $\dt$ be any small positive number and $B_\dt$ be any small ball of radius $\dt$ in the unit cotangent bundle. We take $\eps_\dt=\dt^{2/\nu}$. For each $\eps<\eps_\dt$, the points $y_*=\bp/\|\bp\|$ with $\|\bp\|\leq \eps^{-\nu}$ are $\dt^2$-dense on the unit sphere. In the $\T^N$-component, if g.c.d$(\bp)=1$, then the corresponding periodic orbit is $1/\|\bp\|$-dense when projected to $\T^N$. Thus, there must be a periodic orbit intersecting the ball $B_\dt$. 
\qed 

	


\subsection{The nearly integrable case}

In this section, we consider  general nearly integrable Hamiltonian systems \eqref{EqHam}  and prove Theorem \ref{ThmTorus}. 

We shall need certain nondegeneracy of the unperturbed part $h$ of the system \eqref{EqHam}.
\begin{Lm}\label{Lmh} There is a residual set $\cR\subset C^r(B_1)$, such that for each $h\in \cR$, and almost every point $y$ in the set $B_1\cap h^{-1}(1)$, we have that $D^2h(y)$ is nondegenerate. 
\end{Lm}
\begin{proof}
	We first restrict to the set $\cM\subset C^r(B_1)$ of Morse functions such that critical points are isolated and nondegenerate. For each $h\in \cM$, we have (1) either $h^{-1}(1)$ contains critical points, or (2)  it does not. If (1) occurs, in a neighborhood of the critical point, we have $D^2h$ nondegenerate. Thus it is enough to consider the complement of the neighborhood, which is the same as (2). Since we do not have critical point, by the implicit function theorem $h^{-1}(1)$ is now a submanifold of codimension 1. We then apply Proposition \ref{PropGeneric} to $h$ on the submanifold. 
\end{proof}

With the lemma, we next pick $h\in \cR$, and $y\in h^{-1}(1)$ to be a point where $D^2h(y)$ is nondegenerate. We can then find a neighborhood $B$ of $y$ where $Dh|_B$ is a diffeormorphism to its image. We find a point $\bp/q$ in the image of $Dh|_B$ where $\bp\in \Z^N$ and $q\in \R_+$ and denote $y_*=Dh^{-1}(\bp/q)$. We next pick $\eps$ small such that $\|\bp\|,|q|\leq \eps^{-\nu},\ 0<\nu<1/4$. 

We next perform an $\sqrt\eps$ blowup as we did in Section \ref{SSSBlowup}. Using the same notations, we get
$$H(I,\theta)=\bp/q\cdot I+\sqrt\eps(g(I)+Pf(I,\theta)+Qf(I,\theta))$$
where $g(I)=\frac{1}{2}\langle AI,I\rangle+O(\sqrt\eps),\ A=D^2h(y_*) $ and $f(I,\theta)=V(\theta)+O(\sqrt\eps),$ $V=f(y_*,\theta)$.

Thus repeating the proof of Proposition \ref{PropNF}, we get the following normal form. 
\begin{Prop}\label{PropNFTorus}
	There exists $r_0$ sufficiently large such that for all $r>r_0$ and   $\|f\|_{C^r}\leq 1$, the following holds. Let $y_*\in h^{-1}(1)$ be such that $D^2h(y_*)$ is nondegenerate, then there exists $\eps_0$ such that for all $0<\eps<\eps_0$, there exist $I_*$ satisfying $|I_*|\leq C\eps^{\frac{1}{2}-\nu}$ and  a symplectic transformation $\phi$ defined on $\ B_{\gamma^{1/3}}(I_*)\times \T^N$ and being  $\eps^{\frac{1}{2}-\nu}$-close to identity in the $C^{2}$ norm,   such that 
	$$H\circ\phi(I,\theta)=\frac{\bp}{q}\cdot I+\sqrt\eps G(I)+\sqrt\eps \bar F(I,\theta)+\sqrt\eps\tilde F(I,\theta). $$
	where we have $\|\tilde F\|_{C^2}\leq \gamma^{1+\beta}$, $\bar F=PV+O(\eps^{\frac{1}{4N}}\gamma)$ in $C^2$ and $G(I)=\frac{1}{2}\langle AI,I\rangle+O(\sqrt\eps) $ in $C^2$ with $DG(I_*)=0$. 
\end{Prop}
With the normal form, the remaining proof is similar to that of Theorem \ref{ThmMain} and \ref{ThmRm}. We next analyze the generic conditions. For given $h\in \cR$ as in Lemma \ref{Lmh}, we construct a sequence of asymptotically residual sets $\cR_n$ as follows. We first find a sequence $y_n\in U$, where $U$ is as in the statement of the theorem, such that $\{y_n\}$ is dense in $U$, and each $Dh(y_n)=\bp_n/q_n$ where $\bp_n\in \Z^N$ and $q_n\in \R_+ $, and the entries of $\bp_n$ has no common divisor except 1.  We rearrange the sequence $\bp_n/q_n$ such that $q_n$ is in increasing order. For each $\bp_n/q_n$, we introduce a set $\cR_n\subset \cS^r$ such that each $f\in \cR_n$ satisfies that $PV\circ M^T=Pf\circ \Phi_M^{-1}(y_n,\bar\vartheta)$ has a nondegenerate critical point where the eigenvalues of the Hessian are bounded away from 0 by $c_n/q_n^{r-2}$, where $c_n$ is chosen as in the proof of Theorem \ref{ThmMain}. Note that here the leading term of $\bar F$, i.e. $PV$, is only a function of $\theta$, thus it does not make a difference if $f$ depends on $y$ or not. With this, we can apply Proposition \ref{PropFinal} to complete the proof. \qed


\section{The KAM normal form}\label{SNF}
In this section, we give the proof of the KAM normal form Proposition \ref{PropNF}. 

\subsection{KAM normal form}

\begin{proof}[Proof of the KAM normal form]

	Throughout the proof, we shall take a small $0<\al<\frac{1}{2(N+1)}$, and iteratively $\eps_n=\eps_{n-1}^{2-\al} q$ with $\eps_1=\eps $.

	Let $I_\star$ be a critical point of $g$, we restrict our attention to a $\rho$-neighborhood of $I_\star$, where we have $|\nabla g(I)|\leq |D^2g(I)|\rho$,\ $\rho=\eps_2/\eps \cdot\Lambda$. 
	
	Inductively, we shall take $\rho_n=\frac{1}{\eps}\eps_{n+1} \Lambda$ and $K_n=\frac{1}{3\Lambda \|D^2 g\|q\eps_{n+1}}$ for some $\Lambda$ large, whenever $\rho_n>\gamma^{1/3}$.  When we have $\rho_{n_0}<\gamma^{1/3}$ and $\rho_{n_0-1}>\gamma^{1/3}$, we shall take $\rho_n=\gamma^{1/3}$ and $K_n=\frac{1}{3\|D^2 g\|q\eps\gamma^{1/3}}$ for all $n\geq n_0.$ For all $n$, we shall have $\eps \rho_n K_n\|D^2g\|\leq \frac{1}{3q}.$
	
	{\bf The first step of iteration.}
	
	We start with the first step of iteration. We write the Hamiltonian as $H=H_0+\eps R$, where we take $H_0=\frac{\bp}{q}\cdot I+\eps g+\eps Pf+\eps Q_<f$ and $R=Q_\geq f$, where $Q_<f$ consists of Fourier modes with $\bk\cdot \bp\neq0$ and $|\bk|<K_1$ and $Q_<f$ those with $\bk\cdot \bp\neq0$ and $|\bk|\geq K_1$. 
	
		We consider a symplectic change of coordinates generated by the time-1 map $\phi_{\eps W}$ of the Hamiltonian flow of $\eps W$ defined on $B_{\rho}(I_\star)\times \T^N$. Then we get
	\begin{equation}\label{EqIterationNekh}
		\begin{aligned}
			\phi_{\eps W}^*H&=H_0\circ\phi_{\eps W}+\eps R\circ\phi_{\eps W}\\
			&=H_0+\eps \{H_0,W\}+\eps^2 \int_0^1 (1-t)\{\{H_0,W\},W\}(\phi^t_{\eps W})\,dt+\eps R\circ\phi_{\eps W}\\
			&=\frac{\bp}{q}\cdot I+\eps g(I)+\eps P f+\underline{\eps Q_<f+\eps\{\frac{\bp}{q}\cdot I+ \eps g(I),W\}}+\eps R\circ\phi_{\eps W}\\
			&+\eps^2 \left(\{ P f+Q_< f,W\}+\frac{1}{2}\int_0^1 (1-t)\{\{H_0,W\},W\}(\phi^t_{\eps W})\,dt\right).
		\end{aligned}
	\end{equation}
We next set the term with an underline to be zero. We can solve this equation by Fourier expansion and get $W(I,\theta)=\sum_{\bk\cdot\bp\neq 0}W_\bk(I)e^{i\bk \cdot\theta}$, where 
	 $W_\bk(I)=\frac{i }{\frac{1}{q}\bk\cdot\bp +\eps \bk\cdot\nabla g(I)} f_\bk(I)$. To estimate the coefficient, we have $|\frac{1}{q}\bk\cdot\bp +\eps \bk\cdot\nabla g(I)|\geq \frac{1}{q}-\eps|\bk|\cdot|\nabla g|>\frac{1}{2q}$, due to the choice of $\rho$ and $K_1$. Note that $$\Big|\nabla\frac{1}{\frac{1}{q}\bk\cdot\bp +\eps \bk\cdot\nabla g(I)}\Big|=\Big|\frac{(\eps \bk\cdot\nabla^2 g(I))}{(\frac{1}{q}\bk\cdot\bp +\eps \bk\cdot\nabla g(I))^{2}}\Big|\leq 2|\bk|\|\nabla^2g\|\eps q$$
	 Similarly, we get $$\Big\|\nabla^k\frac{1}{\frac{1}{q}\bk\cdot\bp +\eps \bk\cdot\nabla g(I)}\Big\|\leq C(k,g)\sum_{1\leq\ell\leq |k|}(\eps q)^\ell|\bk|^\ell.$$
	 Thus, we have $\|W\|_{C^r}\leq C(r)q\|Q_<f\|_{C^r}$. 
	 This gives $$\|\phi_{\eps W}-\mathrm{id}\|_{C^{r-1}}\leq C\|\eps W\|_{C^r}\le C\eps q\|f\|_{C^r}.$$ 

We next denote the term in \eqref{EqIterationNekh} with a big parentheis by $\eps_2T^{(1)}$. We have the estimate 
\begin{equation}\label{EqT0}
\eps_2\|T^{(1)}\|_{C^{r-2}}\leq C\eps^2 q(\|f\|_{C^r}^2+\eps q\|f\|_{C^r}^3).	
\end{equation}
Readers can either prove it directly or refer to Section \ref{SSProofLemma} later. 

{\bf The second and later steps of iterations.}

We next introduce 
\begin{equation}\label{EqSplit}
	\begin{aligned}
		 T^{(1)}=\langle T^{(1)}\rangle+PT^{(1)}+Q_<T^{(1)}+Q_\geq T^{(1)}.	
	\end{aligned}
\end{equation}
		where $Q_<$ now means projection to the Fourier modes with $|\bk|\leq K_2$.
	
	We next denote  
	\begin{equation}\label{EqgPQ}
		\begin{aligned}
	\eps g^{(2)}&=\eps g+\eps_2 \langle T^{(1)}\rangle+\eps\langle R\circ\phi_{\eps W}\rangle,\quad \eps Pf^{(2)}=\eps Pf+ \eps_2 P T^{(1)}+\eps P (R\circ\phi_{\eps W}),\\
	\eps_2 Q_<f^{(2)}&=\eps_2 Q_< T^{(1)}+\eps Q_<(R\circ\phi_{\eps W}),\quad \eps R^{(2)}=\eps_2 Q_\geq T^{(1)}+\eps Q_\geq (R\circ\phi_{\eps W}).		
		\end{aligned}
		\end{equation}
		 Then we can write the Hamiltonian as  $$H^{(2)}=H_0^{(2)}+\eps R^{(2)},\quad \mathrm{where}\ H_0^{(2)}=\frac{\bp}{q}\cdot I+\eps g^{(2)}+\eps P f^{(2)}+\eps_2 Q_<f^{(2)},$$ 

	We shall relocate $I_\star$ to be the critical point $I_\star^{(2)}$ of $g^{(2)}$. Note that we have $\|\eps_2 \langle T^{(1)}\rangle\|_{C^2}\leq C\eps_2$ by \eqref{EqT0} and by the definition of Fourier norm we have $$\|R\circ\phi_{\eps W}\|_{C^2}\leq C\|R\|_{C^2}\|\phi_{\eps W}\|_{C^2}\leq CK^{-(r-2)}\eps q\ll\eps_2.$$ Thus we get $\|g^{(2)}-g\|_{C^2}\leq C\eps_2$. Thus by the implicit function theorem  $|I_\star^{(2)}-I_\star|<\rho$. We next restrict our attention to a $\rho_2$-neighborhood of $I_\star^{(2)}$. 
	
With this, we repeat the calculation of \eqref{EqIterationNekh} with $H$ replaced by $H^{(2)}$ and $\eps W$ replaced by $\eps_2 W^{(2)}:\ B_{\rho_2}(I_\star^{(2)})\times \T^N\to \R$. 

Iterating the procedure, at step $n$,  we denote the parenthesis term by $\eps_{n} T^{(n-1)}$, and the generator $\eps_{n-1}W^{(n-1)}:\ B_{\rho_n}(I_\star^{(n)})\times \T^N\to \R$, update the meaning $Q_<$ to projection to Fourier modes with $|\bk|\leq K_n$ and $\bk\cdot\bp\neq 0$, similarly for $Q_\geq,$ and get formally 
$$H^{(n)}(I,\theta):=H^{(n-1)}\circ\phi_{\eps_{n-1} W^{(n-1)}}^1=\frac{\bp}{q}\cdot I+\eps g^{(n)}+\eps Pf^{(n)}+\eps_n Q_< f^{(n)}+\eps R^{(n)},$$
where 
\begin{enumerate}
	\item $\eps g^{(n)}=\eps g^{(n-1)}+\eps_n \langle T^{(n-1)}\rangle+\eps\langle R^{(n-1)}\circ\phi_{\eps_{n-1} W^{(n-1)}}\rangle;$
	\item $\eps Pf^{(n)}=\eps Pf^{(n-1)}+\eps_n PT^{(n-1)}+\eps P (R^{(n-1)}\circ\phi_{\eps_{n-1} W^{(n-1)}}) ,$
	\item $\eps_n Q_<f^{(n)}=\eps_n Q_<T^{(n-1)}+\eps  Q_<(R^{(n-1)}\circ\phi_{\eps_{n-1} W^{(n-1)}})$.
	\item $\eps R^{(n)}=\eps_n Q_{\geq}T^{(n-1)}+\eps Q_\geq (R^{(n-1)}\circ\phi_{\eps_{n-1} W^{(n-1)}}).$
\end{enumerate}

{\bf The estimates.}

Before the estimate, we outline the strategy first. The key observation is that in \eqref{EqIterationNekh}, the estimate of the remainder is dominated by $\eps^2\{Q_<f,W\}$, which is roughly $\eps^2q$ in $C^{r-1}$ and inductively, we expect $\eps_n\sim \eps_{n-1}^{2-\al}q$, where we lose $\eps_{n-1}^\al$ to absorb various uniform constants. We have to make sure that the estimate of $\eps^2\{Pf,W\}$ is much smaller than $\eps^2\{Q_<f,W\}$. Indeed, since $Pf$ consists of Fourier modes with $\bk\cdot \bp=0$, which occurs only when $\bk\in \mathrm{span}\{\bk_i,\ i=2,3,\ldots,N\}$, thus $|\bk|\geq cq$ for some $c>0$. By the decay of Fourier modes in $C^r$, we get $\|Pf\|_{C^{r-k}}\leq \frac{C}{(cq)^k}\|f\|_{C^r}.$ Thus, we can make $Pf$ small by choosing a $C^{r-k}$ norm with larger $k$.  The goal is to make $\eps_n$ smaller than $\|Pf^{(n)}\|_{C^{2}}\sim q^{r-2}=\eps^{(r-2)\frac{N}{N+1}}$ among other estimates. The superexponential decay of $\eps_n$ guarantees that we only need to perform $n=O(\log r)$ steps of iterations. 
\begin{Lm}\label{LmEstimate} Defining $r_n=r-(N+1)n+\log_q\eps_{n-1},$ then we have the following estimates for $n>1$ and $r_n>0$
	\begin{enumerate}
			\item $\|g^{(n)}-g^{(n-1)}\|_{C^2}\leq \eps_{n}$, thus $|I_{\star}^{(n)}-I_{\star}^{(n-1)}|\leq C\eps_n$;
	\item $\|Pf^{(n)} \|_{C^{r_n}}\leq \frac{1}{(cq)^{r-r_n}}$ and $Pf^{(n)} =Pf+O(\eps^{\frac{1}{2N}}\gamma)$ in $C^2$;
		\item $\|Q_< f^{(n)}\|_{C^{r_n}}\leq 1;$ 
		\item $\eps\|R^{(n)}\|_{C^{r_n}}\leq \eps_n/2.$
		\end{enumerate}
\end{Lm}
We postpone the proof of the lemma to the next subsection. 

This is enough to complete the proof of the normal form. Indeed, the recursive relation $\eps_n\leq \eps_{n-1}^{2-\al}q$ and the relation $aq^{-\frac{N+1}{N}}<\eps<bq^{-\frac{N+1}{N}}$ give
$$\eps_n\leq \eps^{(2-\al)^n}q^{\frac{(2-\al)^n-1}{1-\al}}\leq (\eps q^{\frac{1}{1-\al}})^{(2-\al)^n} q^{-\frac{1}{1-\al}}\leq (b^c\eps^{1-c})^{(2-\al)^n} (a^{-1}\eps)^{c},$$
or in terms of $q$, we have $$\eps_n\leq \eps^{(2-\al)^n}q^{\frac{(2-\al)^n-1}{1-\al}}\leq b^{(2-\al)^n}q^{(\frac{1}{1-\al}-\frac{N+1}{N})(2-\al)^n}q^{-\frac{1}{1-\al}},$$ 
where $c=\frac{N}{N+1}\frac{1}{1-\al}$ is less than 1 by the choice of $0<\al<\frac{1}{2(N+1)} .$ We see that $\eps_n$ decays superexponentially. We claim that

{\it Claim: there exist $n_*$, $\al$ and $\beta=\beta(N)$ such that $\eps_{n_*}\leq \gamma^{1+\beta}$ and $\eps_{n_*-1}>\gamma^{1-\beta}$.}

{\it Proof of Claim: } Indeed, suppose there is no such $\beta$, i.e. $\eps_{n_*-1}$ is very close to $\gamma\sim q^{-(r-2)}$. After taking $\log_q$, we get $$\log_q\eps_{n_*-1}\simeq (\frac{1}{1-\al}-\frac{N+1}{N})(2-\al)^{n_*-1} -\frac{1}{1-\al}\simeq -r+2.$$ We can thus adjust $\al$ in the interval $(0,\frac{1}{2(N+1)})$, such that $-\log_q\eps_{n_*-1}<(1-\beta)(r-2)$ and $-\log_q\eps_{n_*}>(1+\beta)(r-2)$, where $\beta$ is a constant depending only on $N$. \hskip 1.0in q.e.d. 

With the last claim, we see that $Q_<f^{(n_*)}$ is bounded in $C^{r_{n_*}}$ norm with $$r_{n_*}=r-(N+1)n_*+\log_q \eps_{n_*-1}\geq r-O(\log r)-(r-2)(1-\beta)\geq \frac{\beta}{2} r. $$

We next denote by $G=g^{(n_*)}$, $\bar F=Pf^{(n_*)}$ and $\eps_{n_*}\tilde F=\eps_{n_*} Q_< f^{(n_*)}+\eps R^{(n_*)}.$
By item (3) of Lemma \ref{LmEstimate}, we get $\|\eps_{n_*} Q_< f^{(n_*)}\|_{C^2}\leq \frac{1}{2}\eps_{n_*}$ and by item (4) we have $\|\eps_1 R^{(n_*)}\|_{C^2}\leq \frac{1}{2}\eps_{n_*}.$
Thus, we complete the proof of the normal form. 



\end{proof}

\subsection{Proof of Lemma \ref{LmEstimate}}\label{SSProofLemma}
In this section, we give the proof of Lemma \ref{LmEstimate} used in the proof of Proposition \ref{PropNF}. 
\begin{proof}[Proof of Lemma \ref{LmEstimate}]
	We stick to the notations in the proof of Proposition \ref{PropNF}. We have obtained $\|W\|_{C^r}\leq 2q\|Q_<f\|_{C^r}$. This gives $\|\phi_{\eps W}-\mathrm{id}\|_{C^{r-1}}\leq C\eps q\|Q_<f\|_{C^r}$. Thus we get $\|R\circ\phi_{\eps W}\|_{C^{r-1}}\leq (1+C\eps q)\|R\|_{C^{r-1}}$.
	
	{\it We next prove \eqref{EqT0}. }
	
	{\it Proof of \eqref{EqT0}.}
	Consider first the term $\{ P f+Q_< f,W\}$ in $T^{(1)}$. The term $\{Pf,W\}$ only contributes to $Qf^{(2)}$ since their Fourier modes satisfy $\bk\cdot\bp\neq 0$, and the term $\{Q_<f,W\}$ contributes to all the terms $g^{(2)},Pf^{(2)},Qf^{(2)}$. 
	We have 
	\begin{equation}\label{EqT1}
\eps^2\|\{ P f+Q_< f,W\}\|_{C^{r-1}}\leq 2\eps^2\|f\|_{C^{r}}\|W\|_{C^{r}}\leq \eps^2 q\|f\|_{C^{r}}^2.		
	\end{equation}

	Next, to estimate $\{\{H_0,W\},W\}$, we rewrite
	\begin{equation}\label{EqT2}
	 \eps^2\{\{H_0,W\},W\}= \eps^2\{-Q_<f,W\}+ \eps^2\{\{\eps Pf+\eps Q_< f,W\},W\},		
	\end{equation}
	and estimate each summand in the same way as above. 
	\hskip 2in q.e.d. 	

	We next estimate carefully each item in the statement of the lemma for the first step of iteration, after which we perform induction. 

	First, we consider $g^{(2)}=g+\eps_2\langle T^{(1)}\rangle+\eps \langle R^{(1)}\circ \phi_{\eps W}\rangle $. By the Fourier cutoff $Q_\geq$ and the definition of $K_n$, we get $$\|R^{(1)}\circ \phi_{\eps W}\|_{C^2}\leq \|R^{(1)}\|_{C^2}\| \phi_{\eps W}\|_{C^2}\leq K_1^{-(r-2)}\|R^{(1)}\|_{C^r}\| \phi_{\eps W}\|_{C^{r-1}}\leq C (q\eps_1)^{r-2}\ll \eps_1.$$
	The $C^2$ estimate $\eps_2\langle T^{(1)}\rangle$ is of order $\eps_2$ given by \eqref{EqT0}. Thus we get item (1).
	
	We next consider item (2). We need to estimate contributions from $T^{(1)}$ and $R\circ\phi_{\eps W}$. By the proof of \eqref{EqT0}, there are two summands in $T^{(1)} $ estimated in \eqref{EqT1} and \eqref{EqT2} respectively. For the first one, we get by the paragraph preceding the statement of Lemma \ref{LmEstimate}
	$$\|P\eps^2\{ P f+Q_< f,W\}\|_{C^{r'}}\leq \frac{\eps^2}{q^{r-r'-1}}\|\{ P f+Q_< f,W\}\|_{C^{r-1}}\leq C\eps^2q \frac{1}{q^{r-r'-1}}\leq \eps^{-\al}(\eps q)^2 \frac{1}{q^{r-r'}}.$$
	The estimate for the second summand is similar. Thus we get 
	$\|PT^{(1)}\|_{C^{2}}\leq \eps^{-\al}(\eps q)^2 \gamma.$
	We use $\eps^{\frac{1}{2N}}$ to give an upper bound of $\eps^{-\al}(\eps q)^2 $ in the statement. 
We next consider  the term $P(R\circ\phi_{\eps W})$ in \eqref{EqgPQ}. We have $\eps P(R\circ\phi_{\eps W})=\eps PR+\eps^2 P \int_0^1\{R,W\}\circ\phi^t_{\eps W}dt$ and $PR=0$. Note that $\eps^2\|\{R,W\}\circ\phi^t_{\eps W}\|_{C^{r-1}}\leq C\eps^2q\|f\|_{C^r}^2$, which has the same estimate as $\eps_2\|T^{(1)}\|_{C^{r-1}}$. 
Thus, by repeating the above estimate for $PT^{(1)}$,  we get the estimate of item (2) holds also for the term $P(R\circ\phi_{\eps W})$. 

We next consider item (3). For the first step of iteration, the estimate follows from that of \eqref{EqT0} and $R\circ \phi_{\eps W}$, thus is similar to item (2). The only difference is that we do not get negative powers of $q$ by considering a $C^{r'}$ norm with $r'<r$, since we the projection $Q_<$ contains all the low frequency terms $|\bk|<\min\{K_1,q\}$. 

	We next consider item (4). This is given by the Fourier cutoff $Q_\geq$ and the definition of $K_n$. We have $\|R^{(n)}\|_{C^s}\le C K_n^{s-t}\|R^{(n)}\|_{C^t},\ s<t,$ and $$K_n^{s-t}\leq (C\eps_{n+1}q)^{t-s}\leq (C(\eps_{n}q)^{2}\eps_n^{-\al})^{t-s}\leq (C\eps_{n}^{\frac{2}{N+1}-\al})^{t-s}\leq \eps_n$$ whenever $t-s\geq N+1$, and we always have $r_{n+1}-r_n\geq N+1$ by definition of $r_n$. 

The passage from step $n-1$ to $n$ is completely analogous. Indeed, we only need to take care of the contributions from $R^{(n)}\circ\phi_{\eps_n W^{(n)}}$ and $T^{(n)}$. The former is done in the last paragraph. We next show by induction how to estimate $T^{(n)}$ by choosing the function space $C^{r_n}$. In place of \eqref{EqT1}, we have 
\begin{equation*}
	\begin{aligned}
&\|\{ \eps P f^{(n-1)}+\eps_{n-1}Q_< f^{(n-1)},\eps_{n-1} W^{(n-1)}\}\|_{C^{r_{n}}}\\
&\leq \eps\eps_{n-1}\|P f^{(n-1)}\|_{C^{r_{n}+1}}\|W^{(n-1)}\|_{C^{r_{n}+1}}+\eps_{n-1}^2\|Q_< f^{(n-1)}\|_{C^{r_{n}+1}}\|W^{(n)}\|_{C^{r_{n}+1}}\\
&\leq\eps\eps_{n-1}\frac{1}{(cq)^{r-r_{n}-1}}q\|Q_<f^{(n-1)}\|_{C^{r_{n}+1}}+\eps_{n-1}^2q\|Q_< f^{(n-1)}\|_{C^{r_{n}+1}}^2\\
	\end{aligned}
\end{equation*}
The second term on the RHS would be bounded by $\eps_n/2$ using $\eps_n=\eps^{2-\al}_{n-1}q$. For the first term, we use the definition of $r_n=r-n(N+1)+\log_q \eps_{n-1}$ to bound the first term also by $\eps_n/2$. The other terms in $T^{(n)}$ are estimated similarly. 
These estimates are then sufficient to conclude the proof of the lemma inductively. 
\end{proof}

\appendix
\section{Abraham transversality theorem}\label{AppAbraham}
The proof of Proposition \ref{PropGeneric} is based on the following parametric transversality theorem of Abraham (see page 48 of \cite{ARK68} and \cite{Ab}). 

Suppose we are given the following data. 
\begin{enumerate}
	\item 
	Let $A$ be a topological space with Baire property, and $M_1$ and $M_2$ be second countable finite dimensional manifolds, $K\subset M_1$ a subset, $V\subset M_2$ a submanifold, and $F: \ A\to C^1(M_1,M_2)$ a map.
	\item $F$ is called a $C^1$ \emph{pseudo-representation} if the evaluation map $ev(TF):\ A\times TM_1\to TM_2$ via $ev(TF)(a,q)=T(Fa)q$ is continuous. 
	\item A $C^1$ pseudo-representation $F$ is called $C^r$-\emph{pseudo-transverse} to $V$ on $K$ if there exists a dense subset $D\subset A$ such that for each $a\in D$, there exist an open subset $B_a$ in a separable Banach space, a continuous map $\psi_a:\ B_a\to A$ and $a'\in B_a$ such that 
	\begin{enumerate}
		\item $\psi_a(a')=a$,
		\item the evaluation map $ev(F\psi_a):\ B_a\times M_1\to M_2$ is $C^r$ and transverse to $V$ on $a'\times K$.
	\end{enumerate}
\end{enumerate}
\begin{Thm}\label{ThmAbraham} Assume that $F:\ A\to C^1(M_1,M_2)$ is $C^r$ pseudo-transverse to $V$ on $K$ with 
	$$r\geq \max\{1,1+\mathrm{dim}M_1-\mathrm{codim}V\}.$$
	Let $R=\{a\in A\ |\ F(a) \mathrm{\ is\ transverse\ to\ } V \mathrm{\ at\ } K\}$.
	If $K=M_1$, then $R$ is residual in $A$, and if $V$ is a closed submanifod and $K\subset M_1$ is compact, then $R$ is open and dense in $A$. 
\end{Thm}
This theorem is a bit abstract. It says that if the space $A$ is so large that the perturbations can be
made inside it in terms of $\psi_a$, then the points in $a\in A$ such that $F_a$ is transverse to $V$ is a residual in $A$. Notice also that $A$ needs not to be a Banach space but only a Baire space. This is useful when $M$ is noncompact or $r$ is infinity. 

To make it  look less formidable, we mention its finite dimensional counterpart (Theorem 2.7 of \cite{Hi}). 
\begin{Thm}
	Let $A,\ M_1,\ M_2$ be $C^r$ manifolds without boundary, and $V\subset M_2$ be a $C^r$ submanifold. Let $F:\ A\to C^r(M_1,M_2)$ satisfying 
	\begin{enumerate}
		\item the evaluation map $ev(F):\ A\times M_1\to M_2$ via $(a,x)\mapsto F_a(x)$ is $C^r$,
		\item $ev(F)$ is transverse to $V$,
		\item $r\geq \max\{1,1+\mathrm{dim} M_2+\mathrm{dim} V-\mathrm{dim} M_1\}$.
	\end{enumerate}
	Then the set $\{a\in A\ |\ F_a\pitchfork V\}$ is residual. If $V$ is closed, then the set is open.  
\end{Thm}
The proof is to apply the Sard theorem to the map $\pi\circ ev(F)^{-1}:\ V\to A$ where $\pi:\ A\times M_1\to A$ is the projection, noticing that $F_a\pitchfork V$ iff $a$ is a regular value of the map $\pi\circ ev(F)^{-1}$.

We also refer readers to \cite{R1,R2} for appliations of Abraham's transversality theorem to obtain some results on generic Hamiltonian dynamics. 
\section{Reeb dynamics and Hamiltonian dynamics}\label{SRelation}
In this section, we study the relation between the Reeb dynamics of the perturbed contact form and the corresponding Hamiltonian dynamics. This result is known experts and we do not claim any originality. We present it here for readers' convenience. 

In our case, we have a contact manifold $(\partial E_a,\lb)$ as in the introduction and the Reeb vector field $R$ determined by $\iota_R d\lb=0$ and $\iota_R \lb=1$ coincides with the Hamiltonian vector field of the system \eqref{EqHam0} on the energy level $1$. Perturbing $\lb$ into $\lb_\eps:=(1+\eps h)\lb$, the equations $\iota_{R_\eps} d\lb_\eps=0$ and $\iota_{R_\eps} \lb_\eps=1$ determines a Reeb vector field $R_\eps$ on $\partial E_a$ that is $O(\eps)$ close to $R$ in the $C^1$ norm by implicit function theorem. In the following, we shall denote by $H$ the Hamiltonian \eqref{EqHam0}. 

\begin{Prop}
	The Reeb vector field $R_\eps$ is conjugate to the Hamiltonian flow of a Hamiltonian $H_\eps$ that is an $O(\eps)$ perturbation of $H-1$, restricted to the zeroth energy level set $\Sigma$ that is a $O(\eps)$ perturbation of $\partial E_a$. 
\end{Prop}
\begin{proof}

	The equation $\iota_{R_\eps} d\lb_\eps=0$ implies that $\iota_{R_\eps} d\lb_\eps$ is proportional to $dH$ when restricted to $\partial E_a$, since both vanish when applied to a vector in  $T\partial E_a$, and the proportion is $O(\eps)$ close to 1. So we set $$\iota_{R_\eps} d\lb_\eps=(1+g_\eps)d(H-1)=d((1+g_\eps)(H-1))$$ when restricted on $\partial E_a$, where the proportion $1+g_\eps$ is determined by the equation $\iota_{R_\eps} \lb_\eps=1$. 
	
	Since $d\lb_\eps$ is an exact 2-form for each $\eps$, then Moser's trick gives a diffeomorphism $\psi_\eps$ satisfying $\psi_\eps^* d\lb_\eps=d\lb$. Indeed, the diffeomorphism is defined by the differential equation $\frac{d}{d\eps}\psi_\eps=X_\eps\circ\psi_\eps$, where the vector field $X_\eps$ is solved from $\frac{d}{d\eps}\lb_\eps=-\iota_{X_\eps}d\lb_\eps$ by the nondegenderacy of $d\lb_\eps$. The equation $\psi_\eps^* d\lb_\eps=d\lb$ is verified by $$0=\frac{d}{d\eps}\psi_\eps^* d\lb_\eps=\psi_\eps^* (\frac{d}{d\eps}d\lb_\eps+d\iota_{X_\eps}d\lb_\eps+\iota_{X_\eps}dd\lb_\eps). $$
	
	The diffeomorphism $\psi_\eps$ is identity outside a small neighborhood of $\partial E_a$ since $h$ is compactly supported. Denote $\Sigma=\psi_\eps^{-1} (\partial E_a)$. Applying $\psi_\eps^*$ to the equation $\iota_{R_\eps} d\lb_\eps=(1+g_\eps) dH$, we get 
	$$\iota_{D\psi_\eps^{-1} R_\eps} d\lb= d\psi^*_\eps ((1+g_\eps)(H-1)). $$
	Since $d\lb$ is the standard symplectic form on $\R^{2N}$, the equation implies that the Reeb flow $R_\eps$ on $\partial E_a$ is conjugate to the Hamiltonian flow on $\Sigma$ given by the Hamiltonian $\psi^*_\eps ((1+g_\eps)(H-1))$ which is $O(\eps)$ close to $H-1$. 
	
\end{proof}
On the other hand, the next proposition shows how to go from a Hamiltonian flow to a Reeb flow. 
\begin{Prop}
	Suppose the Hamiltonian $H_\eps$ that is an $O(\eps)$-perturbation of $H-1$, then there exists a function $f:\ \partial E_a\to \R$ that is $O(\eps)$ close to 1, such that the  Hamiltonian flow of $H_\eps $ on the zeroth energy level set is conjugate to the Reeb flow on the contact manifold $(\partial_a E, f^2\lb). $
\end{Prop}
\begin{proof}
	We consider the zeroth energy level set $\Sigma$ of $H_\eps$. By the implicit function theorem, it is known that $\Sigma$ that is an $O(\eps)$-perturbation of $\partial E_a$. So we write $\Sigma$ as a radial graph over $\partial E_a$, i.e. $\Sigma=\{f(z) z\ |\ z\in \partial E_a\subset \R^{2N}\}$, which induces a diffeomorphism $\psi: \partial E_a\to \Sigma$ via $z\mapsto f(z)z$, where $f$ is $O(\eps)$ close to 1. Noting that $\lb_z=\frac{1}{2}\omega_0(z,\cdot)$ where $\omega_0$ is the standard symplectic form on $\R^{2N}$, we have 
	\begin{equation*}
		\begin{aligned}
			\psi^*\lb_z|_{\Sigma}(h)=\frac{1}{2}\psi^*\omega_0(f(z)z,(df(z) h) z+f(z)h)=f(z)^2\frac{1}{2} \omega_0(z,h)=f(z)^2\lb_z|_{\partial E_a}(h). 
		\end{aligned}
	\end{equation*}

\end{proof}

\end{document}